\documentclass[12pt,epsfig,amsfonts]{amsart}

\usepackage{hyperref, soul, cancel}
\usepackage{graphicx}

\usepackage{euscript}
\usepackage{amsmath,amssymb,amscd,epsfig}
\usepackage{mathrsfs}
\usepackage{todonotes}

\usepackage{xcolor}

\newcommand{\R}{\mathbb R}
\newcommand{\Int}{\text{Int}\,}
\newcommand{\T}{\mathbb T}
\newcommand{\inte}{\mathbb N}
\newcommand{\N}{\mathbb N}
\newcommand{\Px}{\mathcal{P}}

\newcommand{\Z}{\mathbb Z}

\newtheorem{rmk}{Remark}
\newtheorem{thm}{Theorem}[section]
\newtheorem{prop}[thm]{Proposition}
\newtheorem{lem}[thm]{Lemma}

\newtheorem{cor}[thm]{Corollary}

\newcommand{\J}{J^u}

\begin{document}

\title[Thermodynamics of the Katok Map]{Thermodynamics of the Katok Map}

\author{Y. Pesin}
\address{Department of Mathematics, Pennsylvania State
University, University Park, PA 16802, USA}
\email{pesin@math.psu.edu}
\author{S. Senti}
\address{Instituto de Matematica, Universidade Federal do Rio de Janeiro, C.P. 68 530, CEP 21945-970, R.J., Brazil}
\email{senti@im.ufrj.br}
\author{K. Zhang}
\address{Department of Mathematics,
University of Toronto, Toronto, Ontario, Canada}
\email{kezhang@math.umd.edu}

\date{\today}

\thanks{} 

\subjclass{37D25, 37D35, 37A25, 37E30}

\begin{abstract} 
We effect the thermodynamical formalism for the non-uniformly hyperbolic $C^\infty$ map of the two dimensional torus known as the Katok map (\cite{Kat79}). It is a slowdown of a linear Anosov map near the origin and it is a local (but not small) perturbation. We prove the existence of equilibrium measures for any continuous potential function and obtain uniqueness of equilibrium measures associated to the geometric $t$-potential 
$\varphi_t=-t\log |df|_{E^u(x)}|$ for any $t\in(t_0,\infty)$, $t\neq 1$ where $E^u(x)$ denotes the unstable direction. We show that $t_0$ shrinks to $-\infty$ as the size of the perturbation tends to zero. Finally, we establish exponential decay of correlations as well as the Central Limit Theorem for the equilibrium measures associated to $\varphi_t$ for all values of $t\in (t_0, 1)$. 
\end{abstract}

\maketitle

\section{Introduction}
In 1979 A. Katok introduced in \cite{Kat79} the first example of an area preserving $C^\infty$ diffeomorphism of the two dimensional torus $\mathbb{T}^2$, which is non-uniformly hyperbolic. Katok's construction starts with a linear hyperbolic automorphism $A$ of the torus and proceeds by slowing down trajectories in a small neighborhood of a hyperbolic fixed point. As a result this point becomes neutral thus producing trajectories with zero Lyapunov exponents. One can, however, show that the Lyapunov exponents at almost every point are nonzero (with one being positive and another one negative; see Section 2 for the construction and some basic properties of the Katok map). 

The goal of this paper is to effect the thermodynamical formalism for the Katok map. More precisely, we show (see Theorem 3.2) that there is a number $t_0<0$ such that for every 
$t\in (t_0,1)$ there exists a unique  equilibrium measure $\mu_t$ for the \emph{geometric 
$t$-potential} $\varphi_t(x)=-t\log |df|E^u(x)|$, where $E^u(x)$ is the (one-dimensional) unstable subspace at $x$. Moreover, we prove that $\mu_t$ has exponential decay of correlations and satisfies the Central Limit Theorem (CLT). Furthermore, we show that the number $|t_0|$ can be made arbitrarily large if the size of the slow-down neighborhood is sufficiently small. We emphasize that currently this is one of the only very few examples of non-uniformly hyperbolic diffeomorphisms for which one can obtain a sufficiently complete description of thermodynamics (see \cite{SenTak13, SenTak16, LepRio06, AarDenUrb93}; for a more complete picture of thermodynamics for non-uniformly hyperbolic systems see the survey \cite{ClPes16}).  

A crucial property of the Katok map is that it is topologically conjugate via a homeomorphism to the hyperbolic linear automorphism. In particular, since the map $A$ is expansive, so is the Katok map and hence, it admits an equilibrium measure associated to any continuous potential $\varphi$. Furthermore, the Katok map admits a finite Markov partition and as a result for every H\"older continuous function there exists a unique equilibrium measure. In particular, the Katok map possesses a unique measure of maximal entropy. 

We stress that despite presence of zero Lyapunov exponents, the collection of stable and unstable subspaces, $E^s(x)$ and $E^u(x)$, for the Katok map (whose a priori existence almost everywhere is guaranteed by the nonuniform hyperbolicity) can be extended to  continuous (one-dimensional) distributions on the whole torus. This implies that the function $\varphi_t(x)$ is continuous in~$x$. However, neither the conjugacy homeomorphism nor the geometric $t$-potential $\varphi_t(x)$ are H\"older continuous. These are the main obstacles for building thermodynamics for the Katok map as it is known that even a uniformly expanding map may exhibit phase transitions if the potential looses H\"older continuity at a single point (see \cite{PesZha06}). In this regard we note that there are two equilibrium measures corresponding to the potential $\varphi_1$, the area and the Dirac measure at the origin.\footnote{Due to the entropy formula, the area is the unique equilibrium measure for the potential $\varphi_1$ among all measures of positive entropy.}

Our approach to effect thermodynamics for the Katok map and the geometric $t$-potential is based on showing that this map is a Young's diffeomorphism. This is the class of maps introduced by Young in \cite{You98}, see Section 4 for details. These diffeomorphisms admit a symbolic representation by a tower. The base of the tower can be partitioned into countably many subsets and with respect to this partition the induced map on the base is conjugate to the full Bernoulli shift on a countable set of states. The height of the tower is a (not necessarily first) return time to the base (it is also called the inducing time). In our earlier paper \cite{PesSenZha16} we established, among other results, existence and uniqueness of equilibrium measures for the geometric $t$-potential $\varphi_t(x)$ and proved exponential decay of correlations and the CLT for these measures.

In the case of the Katok map the inducing time is the first return time to the base, so that the induced map is the first return map to the base, see Section 6. 

To show that the Katok map is a Young's diffeomorphism we need a substantially deeper knowledge of the behavior of trajectories of the map than is provided in the original Katok paper \cite{Kat79}. This includes, among other results: 1) sharp estimates of the time a given trajectory spends in the slow-down domain, see Lemma \ref{s2estimate}; 2) sharp estimates on the contraction rates along the stable local curves and the angle between stable and unstable curves when they pass through the slow-down domain, see Lemmas \ref{bad-delta} and \ref{bad-gamma1}; 3) construction of stable and unstable invariant cones with sufficiently small angle, see Lemma \ref{cone_invar1};\footnote{In the original Katok paper it is shown that the cones of angle $\frac{\pi}{4}$ centered around the eigendirections of the matrix $A$ are invariant under the map; this is a substantially simpler statement than Lemma \ref{cone_invar1}.} 4) uniform bounds on the contraction and expansion rates and uniform bounded distortion estimates for the induced map, see Section 6. 

In addition, we establish a crucial property of the tower for the Katok map: the number of partition elements of the base with the same inducing time admits an exponential bound with the exponent strictly less than the metric entropy of the (two-dimensional) Lebesgue measure, see Lemma \ref{h1}. It is this estimate that is instrumental in proving the uniqueness of equilibrium measures as well as showing that these measures have exponential decay of correlations and satisfy the CLT, see Proposition \ref{geom_poten}. 
 
We stress again that the loss of uniform hyperbolicity in the Katok map occurred due to the presence of a neutral fixed point. In the one-dimensional setting an example similar in spirit to the Katok map is the well-known Manneville-Pomeau map (see \cite{ManPom80}). This map admits an invariant measure which is absolutely continuous with respect to the (one-dimensional) Lebesgue measure. This invariant measure may or may not be finite depending on the higher order derivatives of the map in a small neighborhood of the neutral fixed point. In the case where the invariant measure is finite, the thermodynamics of this map is similar, see Remark 2 for more details. 

There are other examples of multi-dimensional non-uniformly hyperbolic systems for which some results on thermodynamics of these systems are known. In particular, Senti and Takahashi \cite{SenTak13, SenTak16} proved a theorem similar to our main Theorem \ref{Katok1} for the H\'enon map at the first bifurcation. Also, Leplaideur and Rios \cite{LepRio06} considered a $C^2$ map of the unit square in $\mathbb{R}^2$ with a fixed hyperbolic point whose stable and unstable separatrices have an orbit of homoclinic tangency. Under certain conditions this map possesses a horseshoe and as shown in \cite{LepRio06} every H\"older continuous potential admits a unique equilibrium measure that gives positive weight to any open set intersecting the horseshoe (see also \cite{Bar13} for related results).
For the thermodynamical formalism of parabolic rational maps of the Riemann sphere see \cite{AarDenUrb93} and references therein.

The structure of the paper is as follows. We first introduce the Katok map. In section \ref{sec:main-results} we state our main results. In section \ref{sec:ind-hyperbolic} we describe Young's diffeomorphisms. This structure yields a coding on which we apply the thermodynamical formalism developed in \cite{PesSen05,PesSen08,PesSenZha16}. In section \ref{sec:additional-properties} we show some additional crucial properties of the Katok map which are mentioned above and which are instrumental to our arguments. In section \ref{sec:inducing-Katok} we express the Katok map as a Young's diffeomorphism. The core  technical arguments are in sections \ref{sec:additional-properties} and \ref{sec:inducing-Katok}. Then in section \ref{sec:proof} we apply the thermodynamics of Young's diffeomorphisms from \cite{PesSenZha16} to effect thermodynamics of the Katok map.

{\bf Acknowledgments.} 
The authors would like to thank the Centre Interfacultaire Bernoulli (CIB), Ecole Polytechnique Federale de Lausanne, Switzerland and the ICERM at Brown University where part of this work was done. They also thank A. Zelerowicz for helpful remarks and the referee of the paper for useful comments. S. Senti was partly supported by the EU Marie-Curie IRSES Brazilian-European partnership in Dynamical Systems (FP7-PEOPLE-2012-IRSES 318999 BREUDS), CAPES PVE 158/2012, CNPq and the Shapiro visitor fund from the Pennsylvania State University. Ya. Pesin is partially supported by NSF grants 1101165 and 1400027. K. Zhang is partially supported by the NSERC DISCOVERY grant, reference number 436169-2013. 

\section{Definition of the Katok map}\label{def-Katok}

Consider the automorphism of the two-dimensional torus 
$\T^2=\R^2/\Z^2$ given by the matrix
$A:=\left(\begin{smallmatrix} 2 & 1\\1 & 1\end{smallmatrix}\right)$. We choose a number 
$0<\alpha<1$ and a function $\psi:[0,1]\mapsto[0,1]$ satisfying:
\begin{enumerate}
\item[(K1)] $\psi$ is of class $C^{\infty}$ everywhere but at the origin;
\item[(K2)] $\psi(u)=1$ for $u\ge r_0$ and some $0<r_0<1$;
\item[(K3)] $\psi'(u)> 0$ and is decreasing for  $0<u<r_0$;
\item[(K4)] $\psi(u)=(u/r_0)^\alpha$ for $0\le u\le\frac{r_0}{2}$.
\end{enumerate}
Let $D_r=\{(s_1,s_2): {s_1}^2+{s_2}^2\le r\}$ where $(s_1,s_2)$ is the coordinate system obtained from the eigendirections of $A$. Let $\lambda>1$ be the largest eigenvalue of 
$A$. Setting $r_1:=2r_0\log\lambda$ we have that
\begin{equation}\label{eq:r1}
D_{r_0}\subset\Int A(D_{r_1})\cap\Int A^{-1}(D_{r_1}).
\end{equation}
Consider the system of differential equations in $D_{r_1}$
\begin{equation}\label{batata10}
\frac{ds_1}{dt}= s_1\log\lambda,\quad \frac{ds_2}{dt}=-s_2\log\lambda.
\end{equation}
Observe that $A$ is the time-$1$ map of the local flow generated by this system.

We slow down trajectories of the flow by perturbing the system \eqref{batata10} in $D_{r_0}$ as follows
\begin{equation}\label{batata2}
\begin{aligned}
\frac{ds_1}{dt}=&\quad s_1\psi({s_1}^2+{s_2}^2)\log\lambda\\
\frac{ds_2}{dt}=&- s_2\psi({s_1}^2+{s_2}^2)\log\lambda.
\end{aligned}
\end{equation}
This system of differential equations generates a local flow. Denote by $g$ the time-$1$ map of this flow. The choices of $\psi$, $r_0$ and $r_1$ (see \eqref{eq:r1}) guarantee that the domain of $g$ contains $D_{r_0}$. Furthermore, $g$ is of class $C^\infty$ in $D_{r_0}\setminus \{0\}$ and it coincides with $A$ in some neighborhood of the boundary $\partial D_{r_0}$. Therefore, the map
\begin{equation}\label{mapG}
G(x)=\begin{cases} A(x) & \text{if $x\in\T^2\setminus D_{r_0}$,}\\
g(x) & \text{if $x\in D_{r_0}$}
\end{cases}
\end{equation}
defines a homeomorphism of the torus $\T^2$, which is a $C^\infty$ diffeomorphism everywhere except at the origin. Since $0<\alpha<1$, we have that
$$
\int_0^1\frac{du}{\psi(u)}<\infty.
$$
This implies that the map $G$ preserves the probability measure
$d\nu=\kappa_0^{-1}\kappa\,dm$ where $m$ is the area and the density $\kappa$ is a positive $C^\infty$ function that is infinite at~$0$ and is defined by
$$
\kappa(s_1,s_2):=\begin{cases} (\psi({s_1}^2+{s_2}^2))^{-1}
&\text{if $(s_1,s_2)\in D_{r_0}$},\\ 1 & \text{otherwise}
\end{cases}
$$
and
\[
\kappa_0:=\int_{\T^2}\kappa\,dm.
\]
We further perturb the map $G$ by a coordinate change $\phi$ in 
$\T^2$ to obtain an area-preserving $C^\infty$ map. To achieve this, define a map $\phi$ in $D_{r_0}$ by the formula
\begin{equation}\label{mapshi}
\phi(s_1,s_2):=\frac{1}{\sqrt{\kappa_0({s_1}^2+{s_2}^2)}}
\bigg(\int_0^{{s_1}^2+{s_2}^2}\frac{du}{\psi(u)}\bigg)^{1/2}
(s_1,s_2)
\end{equation}
and set $\phi=\text{Id}$ in $\T^2\setminus D_{r_0}$. Clearly, $\phi$ is a homeomorphism and 
is a $C^{1+\epsilon}$ diffeomorphism for some $\epsilon>0$ and is a $C^\infty$ diffeomorphism outside the origin.
One can show that $\phi$ transfers the measure $\nu$ into the area and that the map $G_{\T^2}=\phi\,\circ\, G\,\circ\,\phi^{-1}$ is a $C^{1+\epsilon}$ diffeomorphism (for some $\epsilon>0$) and is a $C^\infty$ diffeomorphism outside the origin. It is called the \emph{Katok map} (see \cite{Kat79} and also \cite{BarPes13}). 

The following proposition describes some basic properties of this map.

\begin{prop}\cite{Kat79, BarPes13}\label{smoothH}
The map $G_{\T^2}$ has the following properties:
\begin{enumerate}
\item It is topologically conjugated to $A$ via a homeomorphism $H$.
\item It admits two transverse invariant continuous stable and unstable distributions 
$E^s(x)$ and $E^u(x)$ and for almost every point $x$ with respect to area $m$ it has two non-zero Lyapunov exponents, positive in the direction of $E^u(x)$ and negative in the direction of $E^s(x)$. Moreover, the only invariant measure with zero Lyapunov exponents is the Dirac measure at the origin $\delta_0$.
\item It admits two continuous, uniformly transverse, invariant foliations with smooth leaves which are the images under the conjugacy map of the stable and unstable foliations for $A$ respectively.
\item For every $\varepsilon>0$ one can choose $r_0>0$ such that 
$$
|\int\log |D\,G_{\T^2}|E^u|\,dm-\log\lambda|<\varepsilon.
$$
\item It is ergodic with respect to the area $m$.
\end{enumerate}
\end{prop}

\section{Main Results}
\label{sec:main-results}

Recall that given a continuous map $f$ of a compact metric space $X$ and a \emph{potential} function $\varphi$, an invariant Borel probability measure $\mu_\varphi$ is called an \emph{equilibrium measure} if
$$
h_{\mu_\varphi}(f)+\int_X\varphi d\mu_\varphi =\sup_{\mathcal{M}(f,X)}\{h_\mu(f)+\int_X\varphi d\mu\},
$$
where $\mathcal{M}(f,X)$ is the class of all $f$-invariant ergodic Borel probability measures. The supremum on the right hand side coincides with the topological pressure $P(\varphi)$ of the function $\varphi$. We also recall that $f$ has \emph{exponential decay of correlations} with respect to a measure $\mu\in\mathcal{M}(f,X)$ and a class $\mathcal{H}$ of functions on $X$ if there exists $0<\kappa<1$ such that for any 
$h_1, h_2\in\mathcal{H}$,
$$
\Big |\int h_1(f^n(x))h_2(x)\,d\mu -\int h_1(x)\,d\mu
\int h_2(x)\,d\mu\Big |\le C\kappa^n
$$
for some constant $C=C(h_1,h_2)>0$.

The transformation $f$ satisfies the {\it Central Limit Theorem} (CLT) for a class 
$\mathcal{H}$ functions if for any $h\in\mathcal{H}$, which is not a coboundary (i.e., 
$h\ne g\circ f-g$ for any $g\in\mathcal{H}$), there exists $\sigma>0$ such that
$$
\mu\Bigl\{\frac{1}{\sqrt{n}}\sum_{i=0}^{n-1}(h(f^i(x))-\int
h\,d\mu)<t\Bigr\}\rightarrow\frac{1}{\sigma\sqrt{2\pi}}\int_{-\infty}^t 
e^{-\tau^2/2\sigma^2}\,d\tau.
$$
By Statement 1 of Proposition~\ref{smoothH}, the Katok map $G_{\T^2}$ is topologically conjugated to the hyperbolic total automorphism $A$ which is an expansive map. As an immediate corollary one obtains that  $G_{\T^2}$ admits an equilibrium measure associated to any continuous potential.

Consider the geometric $t$-potential $\varphi_{t}=-t\log|D\,G_{\T^2}|E^u|$. Our goal is to describe the existence, uniqueness and ergodic properties of the equilibrium measures associated to $\varphi_t$.

By the continuity property of $E^u$ (see Statement 2 of Proposition~\ref{smoothH}), 
$\varphi_{t}$ is continuous for all $t$. We hence, obtain
\begin{thm}
For every $t\in\mathbb{R}$ the map $G_{\T^2}$ admits an equilibrium measure associated to $\varphi_t$.
\end{thm}

The following result describes the uniqueness and ergodic properties of the equilibrium measures associated to $\varphi_{t}$.

\begin{thm}\label{Katok1}
Consider the Katok map $G_{\T^2}$ and the geometric $t$-potential $\varphi_{t}$. The following statements hold:
\begin{enumerate} 
\item For any $t_0<0$ one can find $r_0=r_0(t_0)$ such that for every $t_0<t<1$
\begin{itemize}
\item there exists a unique equilibrium measure $\mu_t$ associated to $\varphi_t$;
\item $\mu_t$ has exponential decay of correlations and satisfies the CLT with respect to a class of functions which includes all H\"older continuous functions on $\T^2$; note that uniqueness of $\mu_t$ implies that it is ergodic and since correlations decay, $\mu_t$ is in fact, mixing. 
\end{itemize} 
\item For $t=1$ there exist two equilibrium measures associated to 
$\varphi_1$, namely the Dirac measure at the origin $\delta_0$ and the area $m$.
\item For $t>1$, $\delta_0$ is the unique equilibrium measure associated to 
$\varphi_t$.
\end{enumerate} 
\end{thm}
As an immediate corollary of this theorem we obtain the following result.
\begin{cor}
The map $G_{\T^2}$ has unique measure of maximal entropy which has exponential decay of correlations and satisfies the CLT with respect to a class of functions which includes all H\"older continuous functions on $\T^2$.
\end{cor}
\begin{rmk}
One can show that the measure $\mu_t$ is Bernoulli and that the \emph{pressure function} $P(t):=P(\varphi_t)$ is real analytic in the open interval $(t_0,1)$ (see \cite{ShaZel17}). In addition, one can show that the area $m=\mu_1$ has polynomial decay of correlations (see \cite{PesSenSha17}). 
\end{rmk}

\begin{rmk}
The Katok map is a two-dimensional analog of the well-known Manneville--Pomeau map $x\mapsto x+x^{1+\alpha} \pmod 1$, where $\alpha\in (0,1)$. For this map the origin is a neutral fixed point (as in the case of the Katok map). The thermodynamics of this map is well-understood (see \cite{PolW99, PreSla92, aL93,PolYur01,LSV99,You99,Sar01a, Sar02,hH04}) and is quite similar to our Theorem \ref{Katok1}:
\begin{enumerate} 
\item $t<1$: the pressure function $P(t)$ is real analytic and decreasing and there is a unique equilibrium measure $\mu_t$ for $\varphi_t$; this measure is Bernoulli, has exponential decay of correlations, and satisfies the CLT with respect to the class of H\"older continuous functions;
\item $t=1$: the pressure function $P(t)$ is non-differentiable at $t=1$, and $\varphi_1$ has two ergodic equilibrium measures. One of these is the absolutely continuous invariant probability measure $\mu_1$ and the other is $\delta_0$. The measure $\mu_1$ is Bernoulli and it has polynomial decay of correlations;
\item $t>1$: the unique equilibrium state for $\varphi_t$ is $\delta_0$.
\end{enumerate}
Our Theorem \ref{Katok1} establishes the exponential decay of correlations and the CLT with respect to a class which contains all Hölder continuous potentials for the equilibrium measures $\mu_t$ for a large set of values of $t$.
However, the crucial difference between our result for the Katok map and the results mentioned above for the Manneville-Pomeau map is that we can establish the uniqueness and describe the statistical properties of the equilibrium measures $\mu_t$ on an arbitrarily large albeit {\bf finite} interval 
$(t_0,1)$ with $t_0<0$. It is an open problem whether given $r_0>0$, the number $t_0$ is indeed finite and, if so, whether a phase transition occurs at $t=t_0$.
\end{rmk}

\begin{rmk}
One can show that the measure $\mu_t$ has exponential decay of correlations with respect to functions $h_1$ which are H\"older continuous and functions $h_2$ from $L^\infty(M,\mu_t)$.
\end{rmk}

\section{Young's Diffeomorphisms.}
\label{sec:ind-hyperbolic}

\subsection{Definition of Young's diffeomorphisms} Consider a $C^{1+\epsilon}$ diffeomorphism $f:M\to M$ of a compact smooth Riemannian manifold $M$. Following \cite{You98} we describe a collection of conditions on the map $f$. 

An embedded $C^1$ disk $\gamma\subset M$ is called an \emph{unstable disk} (respectively, a \emph{stable disk}) if for all $x,y\in\gamma$ we have that
$d(f^{-n}(x), f^{-n}(y))\to 0$ (respectively, $d(f^{n}(x), f^{n}(y))\to 0$) as 
$n\to+\infty$. A collection of embedded $C^1$ disks 
$\Gamma^u=\{\gamma^u\}$\label{sb:gamma} is called a \emph{continuous family of unstable disks} if there exists a homeomorphism
$\Phi: K^s\times D^u\to\cup\gamma^u$ satisfying:
\begin{itemize}
\item $K^s\subset M$ is a Borel subset and $D^u\subset \R^d$ is the closed unit disk for some $d<\dim M$;
\item $x\to\Phi|{\{x\}\times D^u}$ is a continuous map from $K^s$ to the space of $C^1$ embeddings of $D^u$ into $M$ which can be extended to a continuous map of the closure $\overline{K^s}$;
\item $\gamma^u=\Phi(\{x\}\times D^u)$ is an unstable disk.
\end{itemize}
A \emph{continuous family of stable disks} is defined similarly.

We allow the sets $K^s$ to be non-compact in order to deal with overlaps which appear in most known examples including the Katok map.

A set $\Lambda\subset M$ has \emph{hyperbolic product structure} if there exists a continuous family $\Gamma^u=\{\gamma^u\}$ of unstable disks 
$\gamma^u$ and a continuous family $\Gamma^s=\{\gamma^s\}$ of stable disks $\gamma^s$ such that
\begin{itemize}
\item $\text{dim }\gamma^s+\text{dim }\gamma^u=\text{dim } M$;
\item the $\gamma^u$-disks are transversal to $\gamma^s$-disks with an angle uniformly bounded away from $0$;
\item each $\gamma^u$-disks intersects each $\gamma^s$-disk at exactly one point;
\item $\Lambda=(\cup\gamma^u)\cap(\cup\gamma^s)$.
\end{itemize}

A subset $\Lambda_0\subset\Lambda$ is called an \emph{$s$-subset} if it has hyperbolic product structure and is defined by the same family $\Gamma^u$ of unstable disks as $\Lambda$ and a continuous subfamily 
$\Gamma_0^s\subset\Gamma^s$ of stable disks. A \emph{$u$-subset} is defined analogously.

Assume the map $f$ satisfies the following conditions:
\begin{enumerate}
\item[(Y1)] There exists $\Lambda\subset M$ with hyperbolic product structure, a countable collection of continuous subfamilies 
$\Gamma_i^s\subset\Gamma^s$ of stable disks and positive integers $\tau_i$, $i\in\N$ such that the $s$-subsets 
\begin{equation}\label{Lambda}
\Lambda_i^s:=\bigcup_{\gamma\in\Gamma^s_i}\,\bigl(\gamma\cap \Lambda\bigr)\subset\Lambda
\end{equation}
are pairwise disjoint and satisfy:
\begin{enumerate}
\item \emph{invariance}: for every $x\in\Lambda_i^s$ 
$$
f^{\tau_i}(\gamma^s(x))\subset\gamma^{s}(f^{\tau_i}(x)), \,\, f^{\tau_i}(\gamma^u(x))\supset\gamma^u(f^{\tau_i}(x)),
$$
where $\gamma^{u,s}(x)$ denotes the (un)stable disk containing $x$;
\item \emph{Markov property}: $\Lambda_i^u:=f^{\tau_i}(\Lambda_i^s)$ is a $u$-subset of $\Lambda$ such that for all $x\in\Lambda_i^s$ 
$$
\begin{aligned}
f^{-\tau_i}(\gamma^s(f^{\tau_i}(x))\cap\Lambda_i^u)
&=\gamma^s(x)\cap \Lambda,\\
f^{\tau_i}(\gamma^u(x)\cap\Lambda_i^s)
&=\gamma^u(f^{\tau_i}(x))\cap \Lambda.
\end{aligned}
$$
\end{enumerate}
\end{enumerate}
\vspace*{.5cm}
\begin{enumerate}
\item[(Y2)] For every $\gamma^u\in\Gamma^u$ one has
$$
\mu_{\gamma^u}(\gamma^u\cap \Lambda)>0, \quad \mu_{\gamma^u}\left((\overline{\Lambda\setminus\cup\Lambda_i^s)\cap\gamma^u}\right)=0,
$$
where $\mu_{\gamma^u}$ is the leaf volume on $\gamma^u$.
\end{enumerate}
For any $x\in \Lambda^s_i$ define the \emph{inducing time} by $\tau(x):=\tau_i$ and the \emph{induced map} $F: \bigcup_{i\in\N}\Lambda_i^s\to\Lambda$ by
$$
F|_{\Lambda_i^s}:=f^{\tau_i}|_{\Lambda_i^s}.$$
\begin{enumerate}
\item[(Y3)] There exists $0<a<1$ such that for any $i\in\N$ we have:
\begin{enumerate}
\item[(a)] For $x\in\Lambda_i^s$ and $y\in\gamma^s(x)$, 
$$
d(F(x), F(y))\le a\, d(x,y);
$$
\item[(b)] For $x\in\Lambda^s_i$ and $y\in\gamma^u(x)\cap \Lambda_i^s$,
$$
d(x,y)\le a \, d(F(x), F(y)).
$$
\end{enumerate}
\end{enumerate}
For $x\in \Lambda$ let $\J f(x)=\det |Df|_{E^u(x)}|$ and  
$\J F(x)=\det |DF|_{E^u(x)}|$ denote the Jacobian of $Df|_{E^u(x)}$ and 
$DF|_{E^u(x)}$ respectively.
\begin{enumerate}
\item[(Y4)] There exist $c>0$ and $0<\kappa<1$ such that:
\begin{enumerate}
\item[(a)] For all $n\ge 0$, $x\in F^{-n}(\cup_{i\in\N}\Lambda^s_i)$ 
and $y\in\gamma^s(x)$ we have
\[
\left|\log\frac{\J F(F^{n}(x))}{\J F(F^{n}(y))}\right|\le c\kappa^n;
\]
\item[(b)] For any $i_0,\dots, i_n\in\inte$, 
$F^k(x),F^k(y)\in\Lambda^s_{i_k}$ for $0\le k\le n$ and  
$y\in\gamma^u(x)$ we have
\[
\left|\log
\frac{ \J F(F^{n-k}(x))}{\J F(F^{n-k}(y))}\right|\le c\kappa^k.
\]
\end{enumerate}
\end{enumerate}
\begin{enumerate}
\item[(Y5)] There exists $\gamma^u\in\Gamma^u$ such that 
$$
\sum_{i=1}^\infty \tau_i \mu_{\gamma^u}(\Lambda_i^s) <\infty.
$$
\end{enumerate}

\subsection{Thermodynamics of Young's diffeomorphisms} For $t\in\mathbb{R}$ consider the family of geometric $t$-potentials
$$
\varphi_t(x):=-t\log|\J f(x)|.
$$
Denote by 
\begin{equation}\label{eq:lambda1}
\log\lambda_1:=\sup_{i\ge 1}\, \sup_{x\in\Lambda^s_i}\ \frac{1}{\tau_i}\log|\J F(x)|
\le\max_{x\in M}\log|\J f(x)|.
\end{equation}
Further, denote by 
\begin{equation}\label{SN}
S_n=\sharp\,\{\Lambda_i^s\colon \tau_i=n\}.
\end{equation}
We say that the tower satisfies the \emph{arithmetic condition} if the greatest common denominator of the set of integers $\{\tau_i\}$ is one.\footnote{The arithmetic condition implies that the map on the tower generated by $f$ is topologically mixing.}

The following result is an application of Theorem 7.1 in \cite{PesSenZha16} to the case when the inducing time is the first return time to the base.
\begin{prop}\label{geom_poten}
Let $f\colon M\to M$ be a $C^{1+\epsilon}$ diffeomorphism of a compact smooth Riemannian manifold $M$ satisfying Conditions (Y1)-(Y5). Assume that the inducing time $\tau$ is the first return time to the base of the tower. Then the following statements hold:
\begin{enumerate}
\item There exists an equilibrium measure $\mu_1$ for the potential $\varphi_1$ which is the unique SRB (Sinai-Ruelle-Bowen) measure.
\item Assume that for some constants $C>0$ and $0<h<h_{\mu_1}(f)$ (where $h_{\mu_1}(f)$ is the metric entropy of $\mu_1$) we have that 
\begin{equation}\label{s_n}
S_n\le Ce^{hn}.
\end{equation} 
Define 
\begin{equation}\label{number-t0}
t_0:=\frac{h-h_{\mu_1}(f)}{\log\lambda_1-h_{\mu_1}(f)}.
\end{equation}
Then for every $t_0<t<1$ there exists a measure $\mu_t\in\mathcal{M}(f,Y)$, where $Y:=\{f^k(x)
\colon x\in\bigcup \Lambda_i^s,\ 0\le k\le \tau(x)-1  \}$,
 which is a unique equilibrium measure for the potential $\varphi_t$.
\item Assume that the tower satisfies the arithmetic condition. Assume also that there is $K>0$ such that for every $i\ge 0$, every $x,y\in\Lambda_i^s$ and any $0\le j\le\tau_i$,
\begin{equation}\label{expansion}
d(f^j(x),f^j(y))\le K\max\{d(x,y), d(F(x),F(y))\}.
\end{equation} 
Then for every $t_0<t<1$ the measure $\mu_t$ has exponential decay of correlations and satisfies the CLT with respect to a class of functions which contains all H\"older continuous functions on $M$.
\end{enumerate}
\end{prop}
\begin{rmk}\label{estim0}
In the proof of \cite[Theorem 7.1]{PesSenZha16} the constant $C$ in \eqref{s_n} was assumed to be $1$, but the proof holds true for any $C>0$. Furthermore, the quantity $\lambda_1$ was defined as $\max_{x\in X}|\J f(x)|$ but the proof carries through with $\lambda_1$ defined as in \eqref{eq:lambda1}.

We claim that $\log\lambda_1\ge h_{\mu_1}(f)$, so that $t_0 <0$. To see this, by the entropy formula,
$$
h_{\mu_1}(f)=\lim_{n\to\infty}\frac{1}{n}\log |\J f^n(x)|,
$$
where $x$ is an arbitrary generic point of the measure $\mu_1$ (and the limit exists and is independent of $x$). Consider the subsequence $n_k(x)=\sum_{j=0}^{k-1}\tau(F^j(x))$. We have that
$$
\log |\J f^{n_k}(x)|=\sum_{j=0}^{k-1}\log |\J F(F^j(x))|\le n_k(x)\log\lambda_1
$$
and the claim follows.
\end{rmk}
\begin{rmk}
The requirements in Statement 3 of Proposition \ref{geom_poten} that the tower satisfies the arithmetic condition and that \eqref{expansion} holds need to be added to Theorem 4.5, Statement 2 of Theorem 4.7, and Statement 2 of Theorem 7.7 in \cite{PesSenZha16} for the conclusion of these statements to hold. 
\end{rmk}
\begin{rmk}
Since $t_0< 0$, Proposition \ref{geom_poten} implies in particular, existence and uniqueness of the measure of maximal entropy and establishes exponential decay of correlations and the CLT for this measure. 
\end{rmk}
\begin{rmk}
The measure $\mu_t$ can be shown to be ergodic and since it has exponential decay of correlations it is also mixing. 
\end{rmk}

\section{Some additional properties of the map $G$}\label{sec:additional-properties} 
In this section we establish some crucial properties of the Katok map $G_{\T^2}$ that are in addition to the basic properties described in Proposition\ref{geom_poten}. Since $G_{\T^2}$ is conjugate to the map $G$ given by \eqref{mapG} we will only consider this map. We begin with the following technical lemma. Recall that the number $\alpha$ in property (K4) of the definition of the function $\psi$ satisfies $0<\alpha<1$.
\begin{lem}\label{var-eq1}
For $s=(s_1,s_2)\in D_{\frac{r_0}{2}}$ let 
$$
d_{i,j}=d_{i,j}(s_1,s_2):=\frac{\partial^2 }{\partial s_i\partial s_j}s_2\psi(s_1^2+s_2^2).
$$ 
Then 
\[
\max_{i,j=1,2}\ \left|d_{i,j}\right|\le \frac{6\alpha}{r_0^\alpha}\, (s_1^2+s_2^2)^{\alpha-\frac12}.
\]
\end{lem}
\begin{proof}
We have 
$$
\begin{aligned}
\frac{\partial}{\partial s_1}\left(s_2\psi(s_1^2+s_2^2)\right)&=\frac{2\alpha}{ r_0^\alpha}\, s_1\,s_2\, (s_1^2+s_2^2)^{\alpha-1},\\
\frac{\partial}{\partial s_2}\left(s_2\psi(s_1^2+s_2^2)\right)&=\frac{1}{r_0^\alpha}\,(s_1^2+s_2^2)^\alpha+\frac{2\alpha}{r_0^\alpha}\, s_2^2\, (s_1^2+s_2^2)^{\alpha-1}.
\end{aligned}
$$
Since $-2\le 2(\alpha-1)\frac{s_1^2}{s_1^2+s_2^2}\le 0$ and $0\le\frac{|s_2|}{\sqrt{s_1^2+s_2^2}}\le 1$, we obtain
$$
\begin{aligned}
|d_{1,1}|&
=\frac{2\alpha}{r_0^\alpha}\, \left|\frac{\partial}{\partial s_1} s_1s_2 (s_1^2+s_2^2)^{\alpha-1}\right| \\
&=\frac{2\alpha}{r_0^\alpha}\,  (s_1^2+s_2^2)^{\alpha-1}|s_2|\left|1+2(\alpha-1)
\frac{s_1^2}{s_1^2+s_2^2}\right|\\
&\le \frac{2\alpha}{r_0^\alpha}\, (s_1^2+s_2^2)^{\alpha-\frac12}.
\end{aligned}
$$
The same argument applies to 
$$
|d_{1,2}|
=\frac{2\alpha}{r_0^\alpha}\, (s_1^2+s_2^2)^{\alpha-1} |s_1|\left|1
 +2(\alpha-1)\frac{s_2^2}{s_1^2+s_2^2}
\right|
$$
and
$$
|d_{2,2}|
=\frac{6\alpha}{r_0^\alpha}\, (s_1^2+s_2^2)^{\alpha-1}|s_2|\left|1 +\frac23(\alpha-1)\frac{s_2^2}{s_1^2+s_2^2}\right|
$$
yielding the desired result.
\end{proof}
Consider the solution $s(t)=(s_1(t),s_2(t))$ of  Equation~\eqref{batata2} with an initial condition $s(0)=(s_1(0),s_2(0))$. Assume it is defined on the {\it maximal} time interval $[0,T]$ for which $G^{-1}(s(0))\notin D_{\frac{r_0}{2}}$ and $G(s(T))\notin D_{\frac{r_0}{2}}$ but $s(t)\in D_{\frac{r_0}{2}}$ for all 
$0\le t\le T$. In particular, $s_1(t)\neq 0$ and $s_2(t)\neq 0$. Setting $T_1=\frac{T}{2}$ we have that $s_1(t)\le s_2(t)$ for all $0\le t\le T_1$ and $s_1(t)\ge s_2(t)$ for all $T_1\le t\le T$. The following statement provides effective lower and upper bounds on the functions $s_1(t)$ and $s_2(t)$. 
\begin{lem}\label{s2estimate}
The following statements hold: 
$$
\begin{aligned}
|s_2(t)|&\ge |s_2(a)|\left(1+2^\alpha\,C_1 \,s_2^{2\alpha}(a)\,(t-a)\right)^{-\frac{1}{2\alpha}},& 0\le  a\le t\le T_1;\\
|s_2(t)|&\le |s_2(a)|\left(1+C_1\,s_2^{2\alpha}(a)\,(t-a)\right)^{-\frac{1}{2\alpha}},& 0\le  a\le t\le T;\\
|s_1(t)|&\ge |s_1(a)|\left(1-C_1\,s_1^{2\alpha}(a)\,(t-a)\right)^{-\frac{1}{2\alpha}},& 0\le  a\le t\le T;\\
|s_1(t)|&\le |s_1(T_1)|\left(1-2^\alpha\,C_1\, s_1^{2\alpha}(T_1)\,(t-T_1)\right)^{-\frac{1}{2\alpha}},& T_1\le t\le T;\\
|s_1(t)|&\le |s_1(b)|\left(1+ C_1\,s_1^{2\alpha}(b)\,(b-t)\right)^{-\frac{1}{2\alpha}},& 0\le t\le b\le T,
\end{aligned}
$$
where $C_1=\frac{2\alpha \log\lambda}{r_0^\alpha}$ is a constant. In particular, 
\begin{equation}\label{T-estimate}
T\le \frac{r_0^\alpha}{\alpha 2^\alpha\log\lambda}\ s_1^{-2\alpha}(T_1).
\end{equation}
\end{lem}
\begin{proof}
Assume $s_1(t)>0$ and $s_2(t)>0$ for all $0\le t\le T$. Equation~\eqref{batata2} with $\psi(u)=(u/r_0)^\alpha$ for 
$0\le u\le\frac{r_0}{2}$ yields for all $0\le t\le T$ and $i=1,2$
\begin{equation}\label{eq:si}
\frac{ds_i(t)}{dt}=(-1)^{i+1}\frac{\log\lambda}{r_0^\alpha}\,s_i(t)\,(s_1^2(t)+s_2^2(t))^{\alpha}.
\end{equation}
Obviously, $s_i^2(t)\le s_1^2(t)+s_2^2(t)$ always holds and therefore
$$
\frac{ds_1(t)}{dt}\ge\frac{\log\lambda}{r_0^\alpha}\,s_1^{2\alpha+1}(t)
$$
and 
$$
\frac{ds_2(t)}{dt}\le-\frac{\log\lambda}{r_0^\alpha}\,s_2^{2\alpha+1}(t).
$$
Integrating between $0\le a\le b\le T$ yields
$$
s_1^{-2\alpha}(b)-s_1^{-2\alpha}(a)\le -C_1\,(b-a)
$$
and
$$
s_2^{-2\alpha}(b)-s_2^{-2\alpha}(a)\ge C_1\,(b-a)
$$
with $C_1=\frac{2\alpha\log\lambda}{r_0^\alpha}$.
The second and third inequalities now follow from these two bounds with $b=t$ (after observing that $s_1(t)$ is assumed positive hence
$T\le \frac{1}{C_1}\ s_1^{-2\alpha}(0)$). The last inequality also follows from the above by taking $a=t$.

For the first and fourth inequalities, recall that $s_1(t)\le s_2(t)$ for $0\le t\le T_1$ and $s_2(t)\le s_1(t)$ for $T_1\le t\le T$. Therefore,
\begin{equation}\label{eq:upper-bound-s}
\begin{aligned}
s_1^2(t)+s_2^2(t)\le 2\, s_1^2(t)\qquad& \mbox{ if }\ T_1\le t\le T\\
s_1^2(t)+s_2^2(t)\le 2\, s_2^2(t)\qquad& \mbox{ if }\ 0\le t\le T_1.
\end{aligned}
\end{equation}
Equation~\eqref{eq:si} now yields
$$
\begin{aligned}
\frac{ds_1(t)}{dt}\le\ 2^\alpha\,&\frac{\log\lambda}{r_0^\alpha}\,s_1^{2\alpha+1}(t)\qquad& \mbox{ if }\ T_1\le t\le T\\
\frac{ds_2(t)}{dt}\ge\ -2^\alpha\,&\frac{\log\lambda}{r_0^\alpha}\,s_2^{2\alpha+1}(t)\qquad& \mbox{ if }\ 0\le t\le T_1.
\end{aligned}
$$
Integration between $a$ and $b$ yields
$$
\begin{aligned}
s_1^{-2\alpha}(b)-s_1^{-2\alpha}(a)&\ge 
&-2^\alpha\,C_1\,(b-a)\qquad& \mbox{ if }\ T_1\le a\le b\le T
\\
s_2^{-2\alpha}(b)-s_2^{-2\alpha}(a)&\le &2^\alpha\, C_1\,(b-a)
\qquad& \mbox{ if }\ 0\le a\le b\le T_1.
\end{aligned}
$$
The first bound follows by taking $b=t$. The fourth inequality follows by taking $a=T_1$ and $b=t$ and observing that $s_1(t)$ is assumed positive hence
$T-T_1\le \frac{1}{2^\alpha\,C_1}\,s_1^{-2\alpha}(T_1)$. The last inequality follows by taking 
$a=t$.
\end{proof}
Consider another solution $\tilde{s}(t)=(\tilde{s}_1(t),\tilde{s}_2(t))$ of Equation \eqref{batata2} satisfying an initial condition 
$\tilde{s}(0)=(\tilde{s}_1(0),\tilde{s}_2(0))$. For $i=1,2$, we set 
$$
\Delta s_i(t)=\tilde{s}_i(t)-s_i(t).
$$
Our goal is to obtain an upper bound for $\Delta s(t)=\tilde{s}(t)-s(t)$.
\begin{lem}\label{bad-delta}
Given $0<\mu<1$, assume that $s_1(t)\neq 0\neq s_2(t)$ and
\begin{enumerate}
\item $\Delta s_2(t)>0$ and $|\Delta s_1(t)|\le\mu\Delta s_2(t)$ for $t\in[0,T]$; 
\item $\Big|\frac{\Delta s_2}{s_2}(0)\Big|<\frac{1-\mu}{72}$.
\end{enumerate}
Then 
\begin{align*}
\Delta s_2(t)&\le\frac{\Delta s_2(0)}{s_2(0)}s_2(t)\left(1+2^\alpha\,C_1\, s_2^{2\alpha}(0)\,t\,\right)^{-\beta},& 0\le t\le T_1;\\
\Delta s_2(t)&\le\frac{\Delta s_2(T_1)}{s_1(T_1)}s_1(t)\left(1-2^\alpha\,C_1\, s_1^{2\alpha}(T_1)\,(t-T_1)\right)^{-\beta}, & T_1 \le t \le T, 
\end{align*}
where $\beta=\frac{1-\mu}{2^{\alpha+2}}$ and $C_1$ is the constant in Lemma~\ref{s2estimate}. In addition, 
$$
\|\Delta s(T)\|\le\sqrt{1+\mu^2}\ \frac{s_1(T)}{s_2(0)}\ \|\Delta s(0)\|. 
$$
\end{lem}
\begin{proof}
We prove the lemma assuming that $s_1(t)$ and $s_2(t)$ are strictly positive. The strictly negative cases follow by symmetry. To simplify notation, set $s_1=s_1(t)$, $s_2=s_2(t)$, $u:=s_1^2+s_2^2$ and 
$\tilde{u}:=\tilde{s}_1^2+\tilde{s}_2^2$. By Equation \eqref{batata2}, we have
\begin{equation}\label{deltas}
\begin{aligned}
\frac{d}{dt}\Delta s_2(t)&=\frac{d}{dt}\tilde{s}_2(t)-\frac{d}{dt}s_2(t)=-(\log\lambda)
\left(\tilde{s}_2\psi(\tilde u)-s_2\psi(u)\right) \\
&=-\log\lambda\left(\frac{\partial}{\partial s_1}\Big(s_2\psi(u)\Big)\Delta s_1+\frac{\partial}{\partial s_2}\Big(s_2\psi(u)\Big)\Delta s_2\right)\\
&\hspace{0.5cm}
-\frac{\log\lambda}{2} \sum_{i,j=1,2}d_{i,j}(\xi_1,\xi_2)(\Delta s_i)(\Delta s_j)
\end{aligned}
\end{equation}
for some $\xi=(\xi_1,\xi_2)$ for which $\xi_i$ lies between $s_i(t)$ and $\tilde{s}_i(t)$ for $i=1,2$. Note that 
$$
\frac{\partial}{\partial s_2}\Big(s_1\psi(u)\Big)=2s_1s_2\psi',
\quad\frac{\partial}{\partial s_2}\Big(s_2\psi(u)\Big)=2s_2^2\psi'+\psi
$$
and hence,
\[
\begin{aligned}
\frac{d}{dt}\left(\frac{\Delta s_2}{s_2}\right)
&=\frac{1}{s_2}\left(\frac{d}{dt}\Delta s_2\right)-\frac{\Delta s_2}{s_2^2}\left(\frac{ds_2}{dt}\right)\\
&=-\log\lambda\left(2\psi'(s_1 \Delta s_1 + s_2 \Delta s_2)
+\frac{\Delta s_2}{s_2}\psi\right)\\
&\hspace{0.5cm}+(\log\lambda)\frac{\Delta s_2}{s_2}\psi
-\frac{\log\lambda}{2} \sum_{i,j=1,2}d_{i,j}(\xi_1,\xi_2)\frac{\Delta s_i\Delta s_j}{s_2}\\
&=-\frac{2\alpha\log\lambda}{r_0^\alpha}\, (s_1^2+s_2^2)^{\alpha-1}(s_1 \Delta s_1 + s_2 \Delta s_2)\\
&\hspace{0.5cm}
-\frac{\log\lambda}{2} \sum_{i,j=1,2}d_{i,j}(\xi_1,\xi_2)\frac{\Delta s_i\Delta s_j}{s_2}. 
\end{aligned}
\]
For $0\le t\le T_1$ we have $0<s_1(t)\le s_2(t)$. Since $|\Delta s_1|\le\mu\Delta s_2$, we find that 
$$
s_1\Delta s_1+s_2\Delta s_2\ge (-s_1\mu+s_2)\Delta s_2\ge (1-\mu)s_2\Delta s_2.
$$
Since $|\Delta s_1|\le \mu\Delta s_2<\Delta s_2$, Lemma~\ref{var-eq1} yields
\begin{equation}\label{eq:dij}
\sum_{i,j=1,2}d_{i,j}(\xi_1,\xi_2)\Delta s_i\Delta s_j
\ge -\frac{24\,\alpha}{r_0^\alpha}\ (\xi_1^2 +\xi_2^2)^{\alpha-\frac12}(\Delta s_2)^2.
\end{equation}
It follows that 
$$
\begin{aligned}
\frac{d}{dt}\left( \frac{\Delta s_2}{s_2}\right)
&\le -(1-\mu) \frac{2\alpha\log\lambda}{ r_0^\alpha}\ (s_1^2+s_2^2)^{\alpha}\frac{s_2^2}{s_1^2+s_2^2}\frac{\Delta s_2}{s_2}\\
&\hspace{0.5cm}+\frac{12 \alpha\log\lambda}{r_0^\alpha}\ s_2^{2\alpha}
\left(\frac{\xi_1^2+\xi_2^2}{s_2^2}\right)^{\alpha-\frac12}\left(\frac{\Delta s_2}{s_2}\right)^2.
\end{aligned}
$$
Using again the fact that $0<s_1(t)\le s_2(t)$ for $0\le t\le T_1$, we obtain that 
$s_2^2\le s_1^2+s_2^2\le 2s_2^2$ and conclude that  
$$
\begin{aligned}
\frac{d}{dt}\left(\frac{\Delta s_2}{s_2}\right)
&\le -(1-\mu)\frac{\alpha\log\lambda}{ r_0^\alpha}\ s_2^{2\alpha}\frac{\Delta s_2}{s_2}\\ 
&\hspace{0.5cm}+ \frac{12\alpha\log\lambda}{r_0^\alpha}\  s_2^{2\alpha} \left(\frac{\xi_1^2+\xi_2^2}{s_2^2}\right)^{\alpha-\frac12}\left(\frac{\Delta s_2}{s_2}\right)^2.
\end{aligned}
$$
Setting $\kappa=\kappa(t)=\frac{\Delta s_2}{s_2}(t)$, we find that
\begin{equation}\label{decreasing-kappa}
\frac{d\kappa}{dt}\le -\frac{\alpha\log\lambda}{r_0^\alpha}s_2^{2\alpha}\kappa\Big(1-\mu - 12\left(\frac{\xi_1^2+\xi_2^2}{s_2^2}\right)^{\alpha-\frac12}\kappa\Big).
\end{equation}
Observe that 
$$
0<s_2\le\xi_2\le\tilde{s}_2=s_2+\Delta s_2, \quad 
\xi_1\le s_1+|\Delta s_1|\le s_2+\mu\Delta s_2.
$$
This implies that
$$
1\le\frac{\xi_2^2}{s_2^2}\le\frac{\xi_1^2+\xi_2^2}{s_2^2}\le (1+\mu\kappa)^2+(1+\kappa)^2<2(1+\kappa)^2.
$$
It follows that
$$
\Bigl(\frac{\xi_1^2+\xi_2^2}{s_2^2}\Bigr)^{\alpha-\frac12}\le 
\begin{cases}
1 & 0<\alpha\le\frac12, \\
(2(1+\kappa)^2)^{\alpha-\frac12} & \frac12\le\alpha<1. 
\end{cases}
$$
Using Assumption (2), it is easy to verify that 
$$
\Big(1-\mu -12\left(\frac{\xi_1^2+\xi_2^2}{s_2^2}\right)^{\alpha-\frac12}\kappa(0)\Big)>\frac{1-\mu}{2}
$$
and Equation \eqref{decreasing-kappa} now yields
$$
\frac{d\kappa}{dt}\Big|_{t=0}<-\frac{(1-\mu)\alpha\log\lambda}{2r_0^\alpha}s_2^{2\alpha}(0) \kappa(0)<0.
$$
Hence, $\kappa(t)$ satisfies  
\begin{equation}\label{kappa-estimate}
0<\kappa(t)<\frac{1-\mu}{72}
\end{equation} 
for all $0\le t<\delta$ for some sufficiently small $\delta$. Therefore, the same argument as above applies yielding
\begin{equation}
\label{eq:smallt}
\frac{d\kappa(t)}{dt}<-\frac{(1-\mu)\alpha\log\lambda}{2r_0^\alpha}s_2^{2\alpha}(t) \kappa(t)<0
\end{equation} 
for all $0\le t<\delta$. Using positivity and continuity of functions $\kappa(t)$ and $s_2(t)$ on the interval $[0,T_1]$, it is easy to see that repeating this argument yields the estimates\eqref{kappa-estimate} and\eqref{eq:smallt} for all 
$0\le t\le T_1$.

By the first inequality in Lemma~\ref{s2estimate}, \eqref{eq:smallt} and Gronwall's inequality, we obtain that
\[  	
\begin{aligned}
\kappa(t)&\le\kappa(0)\exp\left(-\frac{(1-\mu)\alpha\log\lambda}{2r_0^\alpha} 
\int_0^t s_2(\tau)^{2\alpha}\,d\tau\right) \\
&\le\kappa(0)\exp\left(-\frac{(1-\mu)\alpha\log\lambda}{2r_0^\alpha} 
\int_0^t s_2(0)^{2\alpha} (1+C_12^\alpha s_2^{2\alpha}(0)\,\tau)^{-1}\,d\tau\right) \\
&=\kappa(0)\exp\left(-\frac{(1-\mu)\alpha\log\lambda}{2r_0^\alpha} 
\frac{1}{C_12^\alpha}\log(1+C_12^\alpha s_2^{2\alpha}(0)\,t\ )\right)\\ 
&=\kappa(0)\left(1+C_12^\alpha s_2^{2\alpha}(0)\,t\ \right)^{-\frac{1-\mu}{2^{\alpha+2}}}. 
\end{aligned}
\]
This implies the first estimate.

To prove the second estimate, using \eqref{deltas} and arguing as above we obtain 
$$
\begin{aligned}
\frac{d}{dt}\left( \frac{\Delta s_2}{s_1}\right)
&=-\log\lambda \left( 2s_2\psi' \Delta s_1+(2s_2^2\psi' +\psi)\frac{\Delta s_2}{s_1}\right)
-\log\lambda \psi \frac{\Delta s_2}{s_1} \\
&\hspace{0.5cm}-\frac{\log \lambda}{2}\sum_{i,j=1,2}d_{i,j}(\xi_1,\xi_2)
\frac{\Delta s_i\Delta s_j}{s_1}.
\end{aligned}
$$
The assumption that $|\Delta s_1|\le \mu\Delta s_2$ and positivity of 
$s_1, s_2, \psi'$ and $\Delta s_2$ imply that 
$$
\begin{aligned}
\frac{d}{dt}\left( \frac{\Delta s_2}{s_1}\right)
&\le-2\log \lambda (\psi - \mu s_1s_2 \psi')\frac{\Delta s_2}{s_1}\\
&\hspace{0.5cm}-\frac{\log \lambda}{2}\sum_{i,j=1,2}d_{i,j}(\xi_1,\xi_2)
\frac{\Delta s_i\Delta s_j}{s_1}. 
\end{aligned}
$$
We have 
$$
\frac{\psi}{\psi'} -\mu s_1s_2=\frac{1}{\alpha}(s_1^2+s_2^2)-\mu s_1s_2\ge (\frac{1}{\alpha}-\frac{\mu}{2})(s_1^2+s_2^2)= \frac{2-\mu}{2\alpha}(s_1^2+s_2^2).
$$
Together with the bound for $d_{i,j}$ from \eqref{eq:dij} this yields
\begin{equation}
\label{decreasing-kappa2}
\begin{aligned}
\frac{d}{dt}\left( \frac{\Delta s_2}{s_1}\right)\le &- (2-\mu)\frac{\log\lambda}{r_0^\alpha} s_1^{2\alpha}\left(\frac{\Delta s_2}{s_1}\right)\\ 
&+12\alpha\frac{\log\lambda}{r_0^\alpha}s_1^{2\alpha}\left(\frac{\xi_1^2 
+\xi_2^2}{s_1^2}\right)^{\alpha-\frac12}\left(\frac{\Delta s_2}{s_1}\right)^2.
\end{aligned}
\end{equation}
Setting $\chi=\chi(t)=\frac{\Delta s_2}{s_1}(t)$, we find that
\begin{equation}\label{newchi}
\frac{d\chi}{dt}\le -\frac{\log\lambda}{r_0^\alpha}s_1^{2\alpha}\chi\Bigl(2-\mu-12\alpha\left(\frac{\xi_1^2+\xi_2^2}{s_1^2}\Bigr)^{\alpha-\frac12}\chi\right).
\end{equation}
Now note that
$\min\{s_i, \tilde{s}_i\}\le \xi_i\le \max\{s_i, \tilde{s}_i\}$ and $\Delta s_i=\tilde{s_i}-s_i$ for $i=1,2$ and thus
$$
s_i-|\Delta s_i|\le \xi_i\le s_i+|\Delta s_i|.
$$
Since $|\Delta s_1|\le \mu \Delta s_2$, it follows that
$$
\xi_1^2 + \xi_2^2\ge \xi_1^2\ge (s_1 - |\Delta s_1|)^2\ge (s_1-\mu|\Delta s_2|)^2\ge s_1^2\left(1-\mu\frac{|\Delta s_2|}{s_1}\right)^2\ge s_1^2(1-\chi)^2 
$$
and
$$
\frac{\xi_1^2 + \xi_2^2}{s_1^2}
\le(1+\mu\chi)^2+(1+\chi)^2<2(1+\chi)^2
$$
It follows that
$$
\Bigl(\frac{\xi_1^2+\xi_2^2}{s_1^2}\Bigr)^{\alpha-\frac12}\le 
\begin{cases}
(1-\chi)^{2\alpha-1} & 0<\alpha\le\frac12, \\
2^{\alpha-\frac12}(1+\chi)^{{2\alpha-1}} & \frac12\le\alpha<1. 
\end{cases}
$$
Observing that $s_1(T_1) =s_2(T_1)$, we have by the first estimate in the lemma and Assumption (2),
$$ 
0\le \chi(T_1)=\frac{\Delta s_2(T_1)}{s_1(T_1)} = \frac{\Delta s_2(T_1)}{s_2(T_1)} 
\le\frac{\Delta s_2(0)}{s_2(0)} < \frac{1-\mu}{72}. 
$$
Therefore, by Assumption (2), we have that
$$
\Big(2-\mu -12\alpha\left(\frac{\xi_1^2+\xi_2^2}{s_1^2}\right)^{\alpha-\frac12}\chi(T_1)\Big)>
\frac{1-\mu}{2}
$$
and Equation \eqref{newchi} now yields
\begin{equation}\label{eq:chi-estimate}
\frac{d\chi}{dt}\Big|_{t=T_1}<-\frac{(1-\mu)\log\lambda}{2r_0^\alpha}s_1^{2\alpha}(T_1)\chi(T_1)<0.
\end{equation}
Repeating the above argument we conclude that the relations \eqref{eq:chi-estimate} hold for all $T_1\le t\le T$. Therefore, Gronwall's inequality and the fourth inequality in Lemma~\ref{s2estimate} now yield
\[  	
\begin{aligned}
\chi(t)&\le \chi(T_1)\exp\left(-\frac{(1-\mu)\log\lambda}{2r_0^\alpha}  
\int_{T_1}^t s_1^{2\alpha}(\tau)\,d\tau\right) \\
&\le \chi(T_1)\exp\left(-\frac{(1-\mu)\log\lambda}{2r_0^\alpha}  
\int_{T_1}^t s_1^{2\alpha}(T_1) \left(1-C_12^\alpha s_1^{2\alpha}(T_1)(\tau-T_1)\right)^{-1}\,d\tau\right) \\
&=\chi(T_1)\exp\left(-\frac{(1-\mu)\log\lambda}{2r_0^\alpha}\frac{1}{C_12^\alpha}
\log\left(1-C_12^\alpha s_1^{2\alpha}(T_1)(t-T_1)\right)\right) \\
&=\chi(T_1)\left(1-C_12^\alpha s_1^{2\alpha}(T_1)(t-T_1)\right)^{-\frac{1-\mu}{\alpha 2^{\alpha+2}}},
\end{aligned}
\]
thus proving the second estimate. 

To prove the last inequality we apply the last inequality in Lemma~\ref{s2estimate} and Gronwall's inequality to \eqref{eq:chi-estimate} on the interval $[T_1,T]$ and obtain that
\[  	
\begin{aligned}
\chi(T)&\le\chi(T_1)\exp\left(-\frac{(1-\mu)\log\lambda}{2r_0^\alpha}  
\int_{T_1}^T s_1^{2\alpha}(\tau)\,d\tau\right)\\
&\le\chi(T_1)\exp\left(-\frac{(1-\mu)\log\lambda}{2r_0^\alpha}  
\int_{T_1}^T s_1^{2\alpha}(T)\left(1+C_12^\alpha s_1^{2\alpha}(T)(T-\tau)\right)^{-1}d\tau\right)\\
&\le \chi(T_1)\exp\left(\frac{1-\mu}{\alpha 2^{\alpha+2}}
\log\left(\frac{1}{1+C_12^\alpha s_1^{2\alpha}(T)(T-T_1)}\right)\right)\\
&\le \chi(T_1)\,
\left(1+C_12^\alpha s_1^{2\alpha}(T)\,(T-T_1)\right)^{-\frac{1-\mu}{\alpha 2^{\alpha+1}}}\le\chi(T_1).
\end{aligned}
\]
Thus
$$ 
\frac{\Delta s_2(T_1)}{s_2(T_1)} \le \frac{\Delta s_2(0)}{s_2(0)}\quad\mbox{ and } \quad\frac{\Delta s_2(T)}{s_1(T)} \le \frac{\Delta s_2(T_1)}{s_1(T_1)}.  
$$ 
Since $s_1(T_1)=s_2(T_1)$, combining the two inequalities, we get 
$$ 
\Delta s_2(T) \le \frac{s_1(T)}{s_2(0)} \Delta s_2(0). 
$$
Since $|\Delta s_1|\le \mu \Delta s_2$, we have 
$\Delta s_2 \le \|\Delta s\|\le \sqrt{1+\mu^2}\Delta s_2$, and the last estimate in the statement of the lemma follows.  
\end{proof}
For every $x\in \T^2$ we define the following two families of cones
$$
K^+(x):=\{v=(v_1,v_2)\in\R^2\colon |v_2|<\mu|v_1|\},
$$
$$
K^-(x):=\{v=(v_1,v_2)\in\R^2\colon |v_1|<\mu|v_2|\},
$$
where we use the coordinate system in the plane generated by the eigendirections of the matrix $A$. The following result shows that with an appropriate choice of $\mu$ these cone families are invariant under $G$. Observe that in Katok's original construction $\mu=1$ whereas we now require $\mu<1$.
\begin{lem}\label{cone_invar1}
There exists $0 <\mu_0 <1$ such that for all $\mu_0 < \mu <1$ and all $x\in\T^2$
\[
(DG)K^+(x)\subseteq K^+(G(x)), \quad (DG)^{-1}K^-(G(x))\subseteq K^-(x).
\]
\end{lem}
\begin{proof}
We shall only prove invariance of the cone family $K^+(x)$ as invariance of the cone family $K^-(x)$ can be obtained by reversing the time. For every $x$ outside of $D_{r_0}$ the invariance of cone family $K^+(x)$ is obvious. We now prove invariance of the cones in $D_{r_0}$. The variational equations for Equation~\eqref{batata2} are
\begin{equation}\label{eq:variational}
\begin{aligned}
\frac{d\zeta_1}{dt}&=\quad(\log\lambda)((\psi+2s_1^2\psi')\zeta_1
+2s_1s_2\psi'\zeta_2),\\
\frac{d\zeta_2}{dt}&=-(\log\lambda)(2s_1s_2\psi'\zeta_1
+(\psi+2s_2^2\psi')\zeta_2)
\end{aligned}
\end{equation}
This yields the following equation for the tangent $\eta=\zeta_2/\zeta_1$
\begin{equation}\label{tangent}
\frac{d\eta}{dt}=
-2\log\lambda\Big((\psi+(s_1^2+s_2^2)\psi')\eta+s_1s_2\psi'(\eta^2+1)\Big).
\end{equation}
If $u=s_1^2+s_2^2\in [r_0/2, r_0]$ then, since $\psi$ is positive and increasing and since $\psi'$ is positive and decreasing, we have that
$$ 
\frac{\psi}{\psi'}(u)\ge\frac{\psi}{\psi'}(r_0/2)=\frac{(1/2)^\alpha}
{\alpha r_0^{-1}(1/2)^{\alpha-1}}=\frac{1}{2\alpha}r_0\ge\frac{1}{2\alpha}u.
$$
If $u \in (0, r_0/2)$, an explicit computation yields 
\[
\frac{\psi}{\psi'}(u) = \frac{r_0^{-\alpha} u^\alpha}{r_0^{-\alpha} \alpha u^{\alpha-1}} = \frac{1}{\alpha}u > \frac{1}{2\alpha} u. 
\]
For $\eta>0$ this implies that
\[
\frac{d\eta}{dt}
\le -2(\log\lambda)\psi'\left((1+\frac{1}{2\alpha})(s_1^2+s_2^2)\eta +s_1s_2(\eta^2+1)\right). 
\]
Note that
\begin{align*}
	 ( 1+ \frac{1}{2\alpha} )& (s_1^2 + s_2^2) \eta + s_1 s_2 (\eta^2 + 1) \\ 
	& = \left(  (1 + \frac{1}{2\alpha})\eta - \frac12(\eta^2  +1) \right) (s_1^2 + s_2^2) + \frac12(\eta^2 + 1) (s_1+s_2)^2 \\
	& \ge  \left(  \frac{1}{2\alpha} \eta - \frac12(\eta -1)^2 \right) (s_1^2 + s_2^2) := \varphi(\eta)(s_1^2 + s_2^2). 
\end{align*}
Since $\varphi(1)=\frac{1}{2\alpha}>0$, there exists $0<\mu_0<1$ such that 
$\varphi(\eta) > 0$ for all $\eta\in [\mu_0, 1]$. As a result, $\frac{d\eta}{dt}<0$ whenever $\eta = \mu > \mu_0$. 

When $\eta <0$ observe that 
\[
\frac{d\eta}{dt}=
2\log\lambda\Big((\psi+(s_1^2+s_2^2)\psi')|\eta| - s_1s_2\psi'(\eta^2+1)\Big).
\]
An argument similar to the above shows that $\frac{d\eta}{dt} > 0$ when 
$\eta=-\mu<-\mu_0$. This proves that the cones are invariant.
\end{proof}

For any $x\in\T^2$ denote
\begin{equation}\label{bad-gamma}
\gamma(x):=\max_{\stackrel{v,w\in K^+(x)}{\|v\|=\|w\|=1}}
\left\{\frac{\angle(DG(x)v, DG(x)w)}{\angle(v,w)}\right\}
\end{equation}
and $\gamma_j(x):=\gamma(G^j(x))$.
\begin{lem}\label{bad-gamma1}
For $x\in D_{\frac{r_0}{2}}$ we have
$$
\prod_{j=0}^{k}\gamma_j(x)\le (1+C_1s_2(0)^{2\alpha}k)^{-\frac{1}{\alpha}},
$$
where $C_1$ is the constant in Lemma \ref{s2estimate}.
\end{lem}
\begin{proof}
Let $G^j(x)=(s_1(j), s_2(j))\in D_{\frac{r_0}{2}}$ for all $0\le j \le k$. Equation~\eqref{tangent} implies 
\begin{equation}\label{etas1}
\frac{d\eta}{ds_1}=-2\left((\frac{1}{s_1}+\frac{s_1^2+s_2^2}{s_1} \frac{\psi'}{\psi})\eta+s_2\frac{\psi'}{\psi}(\eta^2+1)\right).
\end{equation}
Let $\eta(s_1)=\eta(s_1,s_1(j),\eta_i)$ be the solution of this differential equation with the initial condition $\eta=\eta_i$ at  $s_1=s_1(j)$ where $i=1,2.$ Then
$$
\frac{d(\eta_1-\eta_2)}{ds_1}=-2\frac{1}{s_1}\left(1+\frac{\psi'}{\psi}(s_1^2+s_2^2+s_1s_2(\eta_1+\eta_2))\right)(\eta_1-\eta_2).
$$
For $|\eta|<\mu<1$ we have $\eta_1+\eta_2\ge -2$. Since both $\psi$ and 
$\psi'$ are positive, this fact and Gronwall's inequality applied to solutions with $s_1>0$ yield
\begin{multline*}
\left|\eta\left(s_1(j+1),s_1(j),\eta_1\right)
-\eta\left(s_1(j+1),s_1(j),\eta_2\right)\right|
\\
\le|\eta_1-\eta_2|\exp\left(-2\int_{s_1(j)}^{s_1(j+1)}\frac{1}{s_1}\left(1+\frac{\psi'}{\psi}(s_1-s_2)^2\right)\,ds_1\right)\\
\le|\eta_1-\eta_2|\exp\left(-2\int_{s_1(j)}^{s_1(j+1)}\frac{ds_1}{s_1}\right)
=|\eta_1-\eta_2|\left(\frac{s_2(j+1)}{s_2(j)}\right)^2,
\end{multline*}
where the last line follows from the fact that the trajectory $(s_1(j), s_2(j))$ of the map $G$ is a hyperbola and hence, the product $s_1(j)s_2(j)$ is constant. Similar arguments hold for the case $s_1<0$. Since
$$
\gamma_j(x)\le\max_{\eta_1, \eta_2} \frac{\left|\eta\left(s_1(j),s_1(j-1),\eta_1\right)-\eta\left(s_1(j),s_1(j-1),\eta_2\right)\right|}{|\eta_1-\eta_2|},
$$
we have that
\[
\prod_{j=0}^{k}\gamma_j(x)\le\left(\frac{s_2(k)}{s_2(0)}\right)^2.
\]
Since all iterates $G^j(x)=(s_1(j), s_2(j))\in D_{\frac{r_0}{2}}$ for $0\le j\le k$, the second inequality in Lemma~\ref{s2estimate} applies and the statement follows.
\end{proof}
The following lemma supplements Lemma~\ref{cone_invar1} by providing controls on the time spent by the orbits in $D_{r_0}\setminus D_{r_0/2}$.

\begin{lem}\label{tran-bound}
There exists $T_0>0$ depending only on $\lambda$ and $\alpha$ such that for any solution $s(t)$ of Equation \eqref{batata2} with $s(0) \in D_{r_0}$, 
$$ 
\max\{t\colon s(t)\in D_{r_0}\setminus D_{r_0/2}\} < T_0. 
$$
\end{lem}
\begin{proof}
By Equation \eqref{batata2}, for $s_1\le s_2$, we have 
$$ 
\frac{du}{dt} = 2 \psi \log\lambda (s_1^2 - s_2^2) = - 2\psi \log\lambda (u^2 - 4s_1^2s_2^2)^{\frac12},
$$
where $u=s_1^2 + s_2^2$. For $s_2\le s_1$ this reads 
$$
\frac{du}{dt} = 2\psi \log\lambda (u^2 - 4s_1^2s_2^2)^{\frac12}.
$$
We consider the two cases depending on the value of $s_1s_2$, which is invariant under the flow. 

{\bf Case 1}: $4s_1^2s_2^2 \le r_0^2/8$. Under the assumptions of $s_1 \le s_2$ and $r_0/2 \le u \le r_0$, we have 
$$ 
\frac{du}{dt} \le -2\psi \log\lambda (r_0^2/4 - r_0^2/8)^{\frac12} \le -\frac{1}{\sqrt{2}}\psi\log \lambda r_0\le - 2^{-\alpha -1}\log\lambda r_0,  
$$
where we use the fact that by our assumption 
$\psi(u)\ge\psi(r_0/2)=2^{-\alpha}$. Then starting from $u(0)=r_0$, it takes at most $2^\alpha/ \log\lambda$ time to reach $u=r_0/2$, unless the assumption $s_1 \le s_2$ is violated. In the latter case, by symmetry, the orbit will leave $D_{r_0}$ in at most $2\times 2^\alpha/\log\lambda$ time. 

{\bf Case 2}: $4s_1^2s_2^2 > r_0^2/8$.  Using 
$r_0 \ge s_1^2+ s_2^2 \ge s_2^2$, we have 
$r_0^2/8 < 4s_1^2 s_2^2 \le 4s_1^2 r_0$ and hence, $s_1^2 > r_0/32$. By Equation \eqref{batata2}, 
$$ 
\frac{d}{dt}(s_1^2) = 2s_1^2 \psi \log \lambda > \frac{r_0^2}{16} 2^{-\alpha} \log\lambda. 
$$
In this case $s_1^2$ will increase to $r_0$ in at most 
$16\cdot 2^\alpha/\log\lambda$ time, and the orbit leaves $D_{r_0}$. Similar arguments hold when $s_2\le s_1$.
\end{proof}

\section{The Katok map as a Young's diffeomorphism}\label{sec:inducing-Katok}

\subsection{A tower representation for the automorphism $A$} Consider a finite Markov partition $\tilde{\Px}$ for the automorphism $A$ and let 
$\tilde{P}\in\tilde{\Px}$ be a partition element which does not contain the origin. Given $\delta>0$, we can always choose the Markov partition $\tilde{\Px}$ in such a way that $\text{diam }(\tilde{P})<\delta$ and 
$\tilde{P}=\overline{\Int\tilde{P}}$ for any $\tilde{P}\in\tilde{\Px}$. For a point 
$x\in\tilde{P}$ denote by $\tilde{\gamma}^s(x)$ (respectively, 
$\tilde{\gamma}^u(x)$) the connected component of the intersection of 
$\tilde{P}$ with the stable (respectively, unstable) leaf of $x$, which contains $x$. We say that $\tilde{\gamma}^s(x)$ and $\tilde{\gamma}^u(x)$ are \emph{full length} stable and unstable curves through $x$.

Given $x\in\tilde{P}$, let $\tilde{\tau}(x)$ be the first return time of $x$ to 
$\Int\tilde{P}$. For all $x$ with $\tilde{\tau}(x)<\infty$ denote by
$$
\tilde{\Lambda}^s(x)=\bigcup_{y\in\tilde{U}^u(x)\setminus\tilde{A}^u(x)}\,\tilde{\gamma}^s(y),
$$ 
where $\tilde{U}^u(x)\subseteq\tilde{\gamma}^u(x)$ is an interval containing 
$x$ and open in the induced topology of $\tilde{\gamma}^u(x)$, and 
$\tilde{A}^u(x)\subset\tilde{U}^u(x)$ is the set of points which either lie on the boundary of the Markov partition or never return to the set $\tilde{P}$. Note that the one-dimensional Lebesgue measure in $\tilde{\gamma}^u(x)$ of $\tilde{A}^u(x)$ is zero. One can choose $\tilde{U}^u(x)$ such that
\begin{enumerate}
\item for any $y\in\tilde{\Lambda}^s(x)$ we have $\tilde{\tau}(y)=\tilde{\tau}(x)$;
\item for any $y\in\tilde{P}$ such that $\tilde{\tau}(y)=\tilde{\tau}(x)$ we have 
$y\in\tilde{\Lambda}^s(x)$.
\end{enumerate}
Moreover, the image under $A^{\tilde{\tau}(x)}$ of $\tilde{\Lambda}^s(x)$ is a $u$-subset containing $A^{\tilde{\tau}(x)}(x)$. It is easy to see that for any $x,y\in\tilde{P}$ with finite first return time the sets $\tilde{\Lambda}^s(x)$ and 
$\tilde{\Lambda}^s(y)$ either coincide or are disjoint. Thus we have a countable collection of disjoint sets $\tilde{\Lambda}_i^s$ and numbers $\tilde{\tau}_i$ which give a representation of the automorphism $A$ as a Young's diffeomorphism for which the set 
$$
\tilde{\Lambda}=\bigcup_{i\ge 1}\tilde{\Lambda}_i^s
$$ 
is the base of the tower, the sets $\tilde{\Lambda}_i^s$ are the $s$-sets and the numbers $\tilde{\tau}_i$ are the inducing times. Moreover, the set 
$\tilde{\Lambda}$ has direct product structure given by the full length stable and unstable curves, the $s$-sets $\tilde{\Lambda}_i^s$ are disjoint and so are the corresponding $u$-sets 
$\tilde{\Lambda}_i^u=A^{\tilde{\tau}_i}(\tilde{\Lambda}_i^s)$. It is easy to see that Conditions (Y1), (Y3), (Y4) hold. Condition (Y2) is satisfied since 
$x\in(\overline{\tilde\Lambda\setminus\cup\,\tilde\Lambda_i^s)\cap\tilde\gamma^u}$ implies that either $x$ lies on the boundary of the Markov partition or it never returns to the $\Int\tilde P$. Since the inducing time is the first return time to the base, by Kac's formula, it is integrable. Hence, Condition (Y5) holds too. We shall obtain an exponential bound to the number $S_n$ of $s$-sets with the given inducing time. 
\begin{lem}\label{h1}
There exists $h<h_{\text{top}}(A)$ such that 
$$
S_n\le e^{hn}.
$$
\end{lem}
\begin{proof}
It suffices to estimate the number of sets $\tilde{\Lambda}_i^s$ with a given 
$i$. This number is less than the number of periodic orbits of $A$ that originate in $\tilde{P}$ and have minimal period $\tilde{\tau}_i$. Using the symbolic representation of $A$ as a sub-shift of finite type induced by the Markov partition $\tilde{\mathcal{P}}$, one sees that the latter equals the number of symbolic words of length $\tilde{\tau}_i$ for which the symbol $\tilde{P}$ occurs only as the first and last symbol (but nowhere in between). The number of such words grows exponentially with exponent $h<h_{\text{top}}(A)$ (see \cite{KatHas95}, Corollary 1.9.12 and Proposition 3.2.5).
\end{proof}

\subsection{A tower representation for the Katok map} Applying the conjugacy map $H$, one obtains the element $P=H(\tilde{P})$ of the Markov partition 
$\mathcal{P}=H(\tilde{\mathcal{P}})$. Since the map $H$ is continuous, given $\varepsilon$, there is $\delta>0$ such that $\text{diam }(P)<\varepsilon$ for any $P\in\mathcal{P}$ provided $\text{diam }(\tilde{P})<\delta$. Further we obtain the set 
$\Lambda=H(\tilde{\Lambda})$, which has direct product structure given by the full length stable $\gamma^s(x)=H(\tilde{\gamma}^s(x))$ and unstable 
$\gamma^u(x)=H(\tilde{\gamma}^u(x))$ curves. We thus obtain a representation of the Katok map as a Young's diffeomorphism for which 
$\Lambda_i^s=H(\tilde{\Lambda}_i^s)$ are $s$-sets, 
$\Lambda_i^u=H(\tilde{\Lambda}_i^u)=G_{\T^2}^{\tau_i}(\Lambda_i^s)$ are 
$u$-sets and the inducing times $\tau_i=\tilde{\tau}_i$ are the first return time to $\Lambda$. Note that for all $x$ with $\tau(x)<\infty$ 
$$
\Lambda^s(x)=\bigcup_{y\in U^u(x)\setminus A^u(x)}\,\gamma^s(y),
$$ 
where $U^u(x)=H(\tilde{U}^u(x))\subseteq\gamma^u(x)$ is an interval containing $x$ and open in the induced topology of $\gamma^u(x)$, and 
$A^u(x)=H(\tilde{A}^u(x))\subset U^u(x)$ is the set of points which either lie on the boundary of the Markov partition or never return to the set $P$. Note that the one-dimensional Lebesgue measure in $\gamma^u(x)$ of $A^u(x)$ is zero. 

We further restrict the choice of the partition element $P$. Given $Q>0$, we can take the number $r_0$ in the construction of the Katok map so small and, by refining the Markov partition if necessary, we can choose a partition element $P$ such that 
\begin{equation}\label{partition}
G_{\T^2}^n(x)\notin D_{r_0} \text{ for any } 0\le n\le Q 
\end{equation}
and any point $x$ for which either $x\in P$ or $x\notin G_{\T^2}(D_{r_0})$ while 
$G_{\T^2}^{-1}(x)\in D_{r_0}$.

\begin{prop}
There exists $Q>0$ such that the collection of $s$-subsets $H(\Lambda_i^s)$ satisfies Conditions (Y1)--(Y5). 
\end{prop}
\begin{proof}
Condition (Y1) follow from the corresponding properties for $A$ and the fact that $H$ is a topological conjugacy. Condition (Y2) holds since 
$x\in(\overline{\Lambda\setminus\cup\,\Lambda_i^s)\cap\gamma^u}$ implies that either $x$ lies on the boundary of the Markov partition or it never returns to the $\Int P$. Condition (Y5) follows from Kac's formula, since the inducing time is the first return time. We shall prove Conditions (Y3) and (Y4).

Using the conjugacy map $\phi$ from~\eqref{mapshi} one has 
\begin{equation}\label{eq:Fconjugacy}
F(x)=G_{\T^2}^{\tau(x)}(x)=\phi\circ G^{\tau(x)}\circ\phi^{-1}(x), \quad x\in\bigcup_i\Lambda_i^s\subset P.
\end{equation}
Since $\phi$ is smooth everywhere except the origin and $d\phi(x)$ is bounded from below and above on $P$, to establish Condition (Y3) it suffices to prove it for $G^{\tau(x)}$. 

For any $x\in P$ with $\tau(x)<\infty$ we define the finite collection of positive integers 
$\{n_\ell=n_\ell(x)\colon 0\le \ell\le k=k(x)\}$, called the \emph{itinerary} of $x$, as follows: 
\begin{equation}\label{itin}
0=n_0<n_1<\ldots<n_{2k}<n_{2k+1}=\tau(x)
\end{equation} 
and $G^j(x)\in D_{r_1}$\footnote{Recall that $r_1=(\log\lambda)r_0$ so that \eqref{eq:r1} holds.} if and only if $n_{2\ell-1}\le n<n_{2\ell}$ for some 
$1\le\ell\le k$. 

Given $x,y\in P$, denote $x_n=G^n(x)$ and $y_n=G^n(y)$. If $y$ lies on the stable curve through $x$, then $y_n$ lies on the stable curve through $x_n$. For 
$n_{2\ell}\le n<n_{2\ell+1}$, the latter lies in the stable cone for $A$ at $x_n$ and indeed, is an admissible manifold for $A$. It follows that the segment of the stable curve, which connects $x_n$ and $y_n$, expands uniformly under the linear map $A^{-n}$. Hence, due to the choice of the number $Q$, there exists $0<\gamma<1$ such that
\begin{equation}\label{lineardistance}
d(x_{n_{2l+1}}, y_{n_{2l+1}}) \le \gamma^{n_{2l+1}-n_{2l}} d(x_{n_{2l}}, y_{n_{2l}}) \le \gamma^Q d(x_{n_{2l}}, y_{n_{2l}}).
\end{equation}
	
We now turn to the case $n_{2l+1}\le n < n_{2l+2}-1$. Let 
$[m_l, m_{l+1}]\subset [n_{2l+1}, n_{2l+2}-1]$ be the largest interval (possibly empty) with 
$x_n\in D_{r_0/2}$ for all $n\in [m_l,m_{l+1}]$. By Lemma~\ref{tran-bound}, there exists a uniform $T_0>0$ such that $m_l-n_{2l+1}, n_{2l+2}-m_{l+1}<T_0$. Then there exists a constant $C>0$ such that 
$$ 
d(x_{m_l},y_{m_l}) \le C d(x_{n_{2l+1}}, y_{n_{2l+1}}), \quad d(x_{n_{2l+2}}, y_{n_{2l+2}}) \le C d(x_{m_{l+1}}, y_{m_{l+1}}). 
$$
Furthermore, let $s, \tilde{s}:[m_l, m_{l+1}]\to \R^2$ be the solutions of Equation \eqref{batata2} with initial conditions $s(0)=x_n$ and $\tilde{s}(0)=y_n$ respectively. We apply Lemma~\ref{bad-delta} to these orbits. The first two assumptions are satisfied, since $y_n$ is contained in the stable cone from $x_n$. The third assumption requires $d(x_{m_l}, y_{m_l})$ to be sufficiently small. In view of \eqref{lineardistance} this can be assured if we choose the number $r_0$ in the construction of the Katok map sufficiently small to ensure that $Q$ in \eqref{partition} is sufficiently large. Lemma~\ref{bad-delta} implies that 
\begin{equation}\label{itin1} 
d(x_{m_{l+1}}, y_{m_{l+1}})\le\sqrt{1+\mu^2}\ \frac{s_1(m_{l+1})}{s_2(m_l)} d(x_{m_l}, y_{m_l}). 
\end{equation}
Note that $\frac{s_1(m_{l+1})}{s_2(m_l)}$ is uniformly bounded, as both the numerator and denominator are of order $r_0$. This fact and the estimates \eqref{itin1} and \eqref{lineardistance} imply that there exists $0<\theta_1<1$, such that 
\begin{equation}\label{distance}
d(x_{n_{2l+2}}, y_{n_{2l+2}}) \le C^2 \gamma^Q  \frac{s_1(m_{l+1})}{s_2(m_l)} d(x_{n_{2l}}, y_{n_{2l}}) \le \theta_1 d(x_{n_{2l}}, y_{n_{2l}})
\end{equation} 
and the same holds for the odd indexes. It follows that 
$$ 
d(G^{\tau(x)}x, G^{\tau(x)}y) \le \theta_1^k d(x,y) 
$$
where $k$ is defined by \eqref{itin}. Condition (Y3a) follows. Condition (Y3b) can be proven similarly by considering the inverse map. 

We now prove Condition (Y4a) noting that Condition (Y4b) can be verified in a similar manner.
Using the relation~\eqref{eq:Fconjugacy} and the fact that the conjugacy map $\phi$ is smooth everywhere except at the origin, it suffices to prove the corresponding statement for $G^{\tau(x)}$. We need the following general statement. 
\begin{lem}\label{general-BD}
Let $\{A_n\}$ and $\{B_n\}$, $0\le n\le N$ be two collections of linear transformations of 
$\R^d$ and let $K=K(E,\theta)$ be the cone of angle $\theta$ around a subspace $E$. Assume that 
\begin{enumerate}
\item $A_nK\subset K$;
\item there are numbers $\gamma_n>0$ such that for every 
$v, w\in K$, $\|v\|=\|w\|=1$,
\[
\angle(A_nv,A_nw)\le\gamma_n\angle(v,w);
\]
\item there are numbers $d>0$ and $\delta_n>0$ such that for every 
$v\in K$
\[
\|A_nv-B_nv\|\le d\delta_n\|A_nv\|;
\]
\item there is $c>0$ such that for every $v\in K$
\[
\|A_nv\|\ge c\|v\|.
\]
\end{enumerate}
Then there is $C>0$, which is independent of the collections of linear transformations, such that for every $v,w\in K$,
\[
\left|\log\frac{\|\prod_{n=0}^N\,A_nv\|}{\|\prod_{n=0}^N\,B_nw\|}\right|
\le C\left(d\sum_{n=0}^N\,\delta_n+\angle(v,w)\sum_{n=0}^N\,\prod_{k=0}^n\gamma_k\right).
\]
\end{lem}
\begin{proof}[Proof of the lemma] Set $v^0=v$, $w^0=w$ and
\[
v^n=\prod_{k=0}^{n-1}\,A_kv, \quad w^n=\prod_{k=0}^{n-1}\,B_kw.
\]
For $a_n=\left|\log\frac{\|v^n\|}{\|w^n\|}\right|$ and $\tilde{v}^{n}=\frac{v^{n}}{\|v^{n}\|}$ and 
$\tilde{w}^{n}=\frac{w^{n}}{\|w^{n}\|}$ one has
\[
\begin{aligned}
a_n&=\left|\log\frac{\|A_nv^{n-1}\|}{\|B_nw^{n-1}\|}\right|
=\left|\log\left(\frac{\|A_nv^{n-1}\|}{\|A_nw^{n-1}\|}
\frac{\|A_nw^{n-1}\|}{\|B_nw^{n-1}\|}\right)\right|\\
&\le\left|\log\left(\frac{\|v^{n-1}\|}{\|w^{n-1}\|}
\frac{\|A_n\tilde{v}^{n-1}\|}{\|A_n\tilde{w}^{n-1}\|}\right)\right|
+\left|\log\left(\frac{\|A_nw^{n-1}\|}{\|B_nw^{n-1}\|}\right)\right|\\
&\le a_{n-1}+\left|\log\left(\frac{\|A_n\tilde{v}^{n-1}\|}
{\|A_n\tilde{w}^{n-1}\|}\right)\right|
+\left|\log\left(\frac{\|A_nw^{n-1}\|}{\|B_nw^{n-1}\|}\right)\right|.
\end{aligned}
\]
By the assumptions,
\[
\begin{aligned}
\left|\log\left(\frac{\|A_n\tilde{v}^{n-1}\|}
{\|A_n\tilde{w}^{n-1}\|}\right)\right|
&=\left|\log\left(1+\frac{\|A_n\tilde{v}^{n-1}\|-\|A_n\tilde{w}^{n-1}\|}{\|A_n\tilde{w}^{n-1}\|}\right)\right|\\
\le C'\angle(A_n&\tilde{v}^{n-1},A_n\tilde{w}^{n-1})=C'\gamma_n\angle(v^{n-1},w^{n-1})
\end{aligned}
\]
and 
\[
\begin{aligned}
\left|\log\frac{\|A_nw^{n-1}\|}{\|B_nw^{n-1}\|}\right|
&=\left|\log\frac{\|B_nw^{n-1}\|}{\|A_nw^{n-1}\|}\right| \\
&=\left|\log\left(1+\frac{\|B_nw^{n-1}\|-\|A_nw^{n-1}\|}{\|A_nw^{n-1}\|}\right)\right| \\
&\le C''\frac{\|A_nw^{n-1}-B_nw^{n-1}\|}{\|A_nw^{n-1}\|}\le C''d\delta_{n-1},
\end{aligned}
\]
where $C'>0$ and $C''>0$ are constants independent of the collections $A_n$ and $B_n$. This implies that 
\[
\begin{aligned}
a_N&\le C'\sum_{n=0}^{N-1}\angle(v^n,w^n)
+C''d\sum_{n=0}^{N-1}\delta_n\\
&\le C'\left(\sum_{n=0}^{N-1}\prod_{k=0}^n\,\gamma_k\right)\angle(v,w)+C''d\sum_{n=0}^{N-1}\delta_n.
\end{aligned}
\]
The desired result follows.
\end{proof}
We proceed with the proof of Condition (Y4a). We lift the map $G$ to $\mathbb{R}^2$ viewed as the universal cover of the torus and we endow $\mathbb{R}^2$ with the coordinate system $(s_1,s_2)$. Fix $x\in P$ with $N=\tau(x)<\infty$ and $y\in\gamma^s(x)$. Let $K^+=K^+(x)$ be the cone of angle $\arctan\mu$. By Lemma~\ref{cone_invar1}, $K^+$ is invariant under $G$. For $0\le n\le N$, we set $A_n=DG(G^n(x))$ and 
$B_n=DG(G^n(y))$ and we further define $\gamma_n$ by \eqref{bad-gamma}, $\delta_n$ by 
\begin{equation}\label{good-delta-n}
\delta_n=\frac{1}{d(x,y)}\max_{v\in K^+\setminus\{0\}}\frac{\|A_nv-B_nv\|}{\|A_nv\|}
\end{equation}
and $d=d(x,y)$.

\begin{lem}\label{summable}
The maps $A_n$, $B_n$ and the cone $K^+$ satisfy the conditions of Lemma~\ref{general-BD}. Furthermore, there exist constants $\tilde{C}>0$ and $0<\theta_2 <1$ independent of the choice of $x$ such that $\delta_n$ and $\gamma_n$ satisfy
\begin{equation}\label{eq:summable}
\sum_{n=0}^{\tau(x)-1} \delta_n <\tilde{C}, \quad \sum_{n=0}^{\tau(x)-1}\prod_{j=0}^n\gamma_j < \tilde{C}, \quad \prod_{n=0}^{\tau(x)-1}\gamma_n <\theta_2.
\end{equation}
\end{lem}
We first show how to derive Property (Y4a) from Lemma~\ref{summable}. For 
$y\in\gamma^s(x)$ and two vectors $v\in K^+(x)$ and $w\in K^+(y)$ Lemmas 
\ref{general-BD} and ~\ref{summable} yield
$$
\begin{aligned} 
\left|\log\frac{\|DG^{\tau(x)}(x)v\|}{\|DG^{\tau(x)}(y)w\|}\right|
&=\left|\log\frac{\|\prod_{n=0}^{\tau(x)-1}\,A_nv\|}{\|\prod_{n=0}^{\tau(x)-1}\,B_nw\|}\right|\\
&\le C(C' d(x,y)+C''\angle (v,w)). 
\end{aligned} 
$$
Furthermore,
\begin{equation}\label{eq:angleratio}
\frac{\angle (DG^{\tau(x)}(x)v, DG^{\tau(x)}(y)w)}{\angle(v,w)} \le \prod_{n=0}^{\tau(x)-1}\gamma_n < \theta_2. 
\end{equation}
Assume that $v^n\in E^u((G^{\tau(x)})^n(x))$ and $w^n\in E^u((G^{\tau(x)})^n(y))$. Then there exists $v\in E^u(x), w\in E^u (y)$ such that $v^n=D(G^{\tau(x)})^n(x)v$ and $w^n=D(G^{\tau(x)})^n(y)w$. Using \eqref{eq:angleratio} and Property (Y3), we have
$$
\begin{aligned}
\left|\log\frac{\|DG^{\tau(x)}((G^{\tau(x)})^n(x))v^n\|}{\|DG^{\tau(x)}((G^{\tau(x)})^n(y))w^n\|}\right| 
&\le CC' d((G^{\tau(x)})^n (x),(G^{\tau(x)})^n(y))\\
&+CC''\angle (D(G^{\tau(x)})^n(x)v,D(G^{\tau(x)})^n(y)w) \\
&\le cCC' \theta_1^n d(x,y) + CC'' \theta_2^n \angle(v,w).
\end{aligned}
$$
By the relation~\eqref{eq:Fconjugacy}, the same ratio for the induced map $F$ differs only by the differential of the conjugacy map $\phi$. Since $x,y\notin D_{r_0}$ and $\phi$ is smooth everywhere except the origin, the above inequality also hold for $F$. Observe that $0<\theta_1,\theta_2<1$ and Property (Y4a) follows. This completes the proof of the theorem modulo Lemma~\ref{summable}. 
\end{proof}

\begin{proof}[Proof of Lemma~\ref{summable}]

In this proof, $C$ refers to an unspecified positive constant that may depend on the H\"older exponent $\alpha$ of the function $\psi$ and constants $\lambda$ and $r_0$, but not on the choice of $x$ (hence, not on $\tau(x)$). We will also use the phrase ``uniformly bounded'' if an expression can be bounded by such a constant. 

We have already shown the invariance of the cone $K^+$ and hence, the first requirement in Lemma~\ref{general-BD} is satisfied. The second and third requirements hold by the choice of $\delta_n$ and $\gamma_n$ (see \eqref{bad-gamma} and \eqref{good-delta-n}). Condition (4) follows from the definition of $A_n$ and the fact that $G$ is a diffeomorphism, and hence, $\|DG(x)\|$ is bounded below by a constant. 

We now proceed with the proof of \eqref{eq:summable}.

\vskip .1in
\noindent\textit{{\bf Part 1}: Estimating $\delta_n$.} 
Denote $x_n=G^n(x)$, $y_n=G^n(y)$. Since $y_n\in \gamma^s(x_n)$, the vector $y_n-x_n$ is contained in the stable cone $K^-$. Due to symmetry, it suffices to consider $x_n$ and $y_n$ in the first quadrant, and assume that $s_2$ component of $y_n$ is larger than the $s_2$ component of $x_n$, i.e., that $\Delta s_2>0$ (interchange $x_n$ and $y_n$ otherwise).

Let $n_0,\dots, n_k$ be the itinerary of $x$, and consider $n_{2l}\le n\le n_{2l+1}-1$. We have that $x_n\notin D_{r_0}$. In this case $A_n=B_n=A$ are constant matrices and hence, $\delta_n =0$. 

We now consider $n\in [n_{2l+1}, n_{2l+2}-1]$ and let 
\begin{equation}\label{eq:D}
D(s_1,s_2) = \log \lambda \begin{bmatrix}
\psi + 2s_1^2 \psi' & 2s_1s_2 \psi' \\
-2s_1s_2 \psi' & \psi + 2s_2^2 \psi'
\end{bmatrix},
\end{equation}
be the coefficient matrix of the variational equation \eqref{batata2}. Let also 
$s(t),\tilde{s}(t):[n,n+1]\to \R^2$ be solutions of \eqref{batata2} with initial conditions $s(n)=x_n$, $\tilde{s}(n)=y_n$. Finally, let $A_n(t), B_n(t)$ be $2\times 2$-matrices solving the variational equations
$$
\frac{dA_n(t)}{dt} = D(s(n+t)) A_n(t), \quad \frac{dB_n(t)}{dt}= D(\tilde{s}(n+t))B_n(t)
$$
with initial conditions $A_n(0)=B_n(0)= Id$. We have 
$$
A_n = dG(x_n) = A_n(1), \quad B_n =dG(y_n) = B_n(1).
$$
Since
$$ 
\begin{aligned}
\frac{dA_n(t)}{dt}-\frac{dB_n(t)}{dt}
&=\bigl(  D(s(n+t))-D(\tilde{s}(n+t))\bigr)A_n(t) \\
&\hspace{0.5cm}+ D(\tilde{s}(n+t))\bigl(A_n(t) -B_n(t)\bigr), 
\end{aligned}
$$
we obtain
$$
A_n(t) -B_n(t) = A_n(t)\int_0^t A_n^{-1}(\tau) \bigl( D(s(n+\tau))-D(\tilde{s}(n+\tau))  \bigr) A_n(\tau) d\tau.
$$
We have $\|D(s) -D(\tilde{s})\| \le \|\partial D(\xi)\|\cdot\| \Delta s\|$, where $\Delta s = \tilde{s}-s$ and $\xi=(\xi_1,\xi_2)$ with $\xi_i\in [s_i,\tilde{s}_i]$ for $i=1,2$. By Lemma~\ref{var-eq1},
\begin{equation}\label{AnBn}
\begin{aligned}
\|A_n - B_n\|&\le \|A_n(1)\| \\ 
& \times\sup_{0 \le \tau \le 1} \Big[\|A_n^{-1}(\tau)\|\|A_n(\tau)\| D(s(n+\tau))-D(\tilde{s}(n+\tau)) \| \Big]\\  
&\le C \sup_{0 \le \tau\le 1} \Big[(\xi_1^2+\xi_2^2)^{\alpha-\frac12}(n+\tau) \|\Delta s(n+\tau)\|\Big].
\end{aligned}
\end{equation}
By assumption (4) of Lemma~\ref{general-BD},
$$
\delta_n \le \frac{1}{cd(x,y)}\|A_n-B_n\| = \frac1c \frac{d(x_{n_{2l+1}},y_{n_{2l+1}})}{d(x,y)} 
\frac{\|A_n-B_n\|}{d(x_{n_{2l+1}}, y_{n_{2l+1}})}.  
$$
Again, recall that $[m_l, m_{l+1}] \subset [n_{2l+1}, n_{2l+2}-1]$ is the largest (possibly empty) interval such that $x_n \in D_{r_0/2}$ for all $n\in [m_l, m_{l+1}]$ and $[m_l, T_l]$ is the largest time interval on which $s_1(t)\le s_2(t)$. If such a $m_l$ does not exist then $x_n\in D_{r_0}\setminus D_{r_0/2}$ for all $n\in [n_{2l+1}, n_{2l+2}-1]$ and Lemma \ref{tran-bound} implies that $n_{2l+2}-n_{2l+1}$ is uniformly bounded.
We now claim that 
\begin{equation}\label{def:mathcalD}
\mathcal{D}_l:=\sum_{n=n_{2l+1}}^{n_{2l+2}}\frac{\|A_n-B_n\|}{d(x_{n_{2l+1}}, y_{n_{2l+1}})}\le C,
\end{equation}
where $C$ is some constant independent of $l$. This implies the summability of $\delta_n$ since, by~\eqref{distance}, we have 
\begin{multline*}
\sum_{n=0}^{\tau(x)-1}\delta_n=\sum_{l=1}^k \sum_{n=n_{2l+1}}^{n_{2l+2}} \delta_n  \\ 
\le \sum_{l=1}^k\frac1c \frac{d(x_{n_{2l+1}},y_{n_{2l+1}})}{d(x,y)}\sum_{n=n_{2l+1}}^{n_{2l+2}}\frac{\|A_n-B_n\|}{d(x_{n_{2l+1}}, y_{n_{2l+1}})} \le C\frac1c \sum_{l=1}^k \alpha_1^l. 
\end{multline*}
We now prove the estimate~\eqref{def:mathcalD}. We have that 
$$
\mathcal{D}_l= \left( \sum_{n=n_{2l+1}}^{m_l-1} + \sum_{n=m_l}^{T_l-1} + \sum_{n=T_l}^{m_{l+1}-1} + \sum_{n=m_{l+1}}^{n_{2l+2}}\right)  \frac{\|A_n-B_n\|}{d(x_{n_{2l+1}}, y_{n_{2l+1}})} 
$$
and we shall show that each of the four sums is uniformly bounded. To this end, observe that by Lemma~\ref{tran-bound}, $m_l-n_{2l+1}\le T_0$ and $n_{2l+2}-m_{l+1} \le T_0$, and hence, each sum $\sum_{n=n_{2l+1}}^{m_l-1}$ and $ \sum_{n=m_l}^{n_{2l+2}}$ involves at most $T_0$ terms of uniformly bounded quantity, and hence, is itself uniformly bounded. 

Observe that since $\tilde{s}$ is contained in the stable cone at $s$, we always have
\begin{equation}\label{eq:Delta}
|\Delta s_1|<\mu\Delta s_2\le\Delta s_2.
\end{equation} 

{\bf Case 1: $n\in [m_l, T_l]$.} We have $s_1 \le s_2$. To apply Lemma~\ref{bad-delta} to the time interval $[m_l, n]$ we need to ensure that 
$\frac{\Delta s_2(m_1)}{s_2(m_1)}<\frac{1-\mu}{72}$. This can be guaranteed by choosing the number $r_0$ in the construction of the Katok map sufficiently small so that the number $Q$ in \eqref{partition} is sufficiently large. Using \eqref{eq:Delta} and Lemma~\ref{bad-delta} for $n\le T_l-1$ and $0\le\tau\le 1$, we have 
\begin{equation}
\label{eq:delta-control}
\begin{aligned} 
\|\Delta s(n+\tau)\|&\le 2\Delta s_2(n+\tau)\\
&\le 2\Delta s_2(m_l)\frac{s_2(n+\tau)}{s_2(m_l)}(1+C_1 s_2^{2\alpha}(m_l) (n+\tau-m_l))^{- \beta}.
\end{aligned}
\end{equation}
As in the proof of Lemma ~\ref{bad-delta} we have that 
$$
s_2^2(t)\le \xi_1^2(t)+\xi_2^2(t)\le 2(1+\kappa)^2s_2^2(t)
\le Cs_2^2(t)
$$
and hence,
$$
(\xi_1^2(n+\tau)+\xi_2^2(n+\tau))^{\alpha-\frac12}\le C \,s_2^{2\alpha -1}(n+\tau).
$$
Now \eqref{AnBn} and \eqref{eq:delta-control} and the fact that $|\Delta s_2|\le \|\Delta s\|$ yield
\begin{multline*}
\|A_n-B_n\|\le\\
\le C \frac{\|\Delta s(m_l)\|}{s_2(m_l)}
\sup_{0\le\tau\le 1}\Big[ 
s_2^{2\alpha}(n+\tau)
\big(1+C_1s_2^{2\alpha}(m_l) (n+\tau-m_l)\big)^{- \beta} \Big].
\end{multline*}
Applying Lemma~\ref{s2estimate} on the time interval $[m_l, n+1]$ we get
\begin{multline*}
\|A_n-B_n\| \le\\
 C \frac{\|\Delta s(m_l)\|}{s_2(m_l)}
\sup_{0\le\tau\le 1}\Big[s_2(m_l)^{2\alpha} (1+C_1s_2^{2\alpha}(m_l)(n+\tau-m_l))^{-1 - \beta}\Big]\\
\le C \|\Delta s(m_l)\| s_2(m_l)^{2\alpha-1} (1+C s_2^{2\alpha}(m_l)(n-m_l))^{-1 - \beta}.
\end{multline*}
Note that  $s_2(m_l)$ is bounded from above and below by a multiple of $r_0$ ($m_l$ is the first time larger than $n_{2l+2}$ that the orbit enters $D_{r_0/2}$), and 
$|\Delta s(m_l)|=d(x_{m_l}, y_{m_l})$. Moreover, since $n_{2l+1}-m_l$ is uniformly bounded in $l$, the ratio 
$\frac{d(x_{m_l}, y_{m_l})}{d(x_{n_{2l+1}}, y_{n_{2l+1}})}$ is uniformly bounded in $l$.
  
We conclude that 
$$
\frac{\|A_n-B_n\|}{d(x_{n_{2l+1}}, y_{n_{2l+1}})} \le C (1+C_1s_2^{2\alpha}(m_l)(n-m_l))^{-1 - \beta}
$$
and hence,
$$
\sum_{n=m_l}^{T_l-1}\frac{\|A_n-B_n\|}{d(x_{n_{2l+1}}, y_{n_{2l+1}})} \le \sum_{n=m_l}^{T_l-1} C (1+C_1s_2^{2\alpha}(m_l)(n-m_l))^{-1 - \beta}
$$ 
is uniformly bounded in $l$. 

{\bf Case 2: $n\in [T_l, m_{l+1}]$.} We have $s_1\ge s_2$. Due to symmetry of the system, we have $T_l \ge (m_{l+1}-m_l-2)/2$ (the additional $2$ coming from possible round-off to an integer). 

By \eqref{eq:chi-estimate}, \eqref{eq:Delta} and Lemma~\ref{bad-delta} for all $T_l\le t\le m_{l+1}$  and $0\le \tau\le 1$ we have
$$
\begin{aligned}
\|\Delta s(n+\tau)\|& \le 2\Delta s_2(n+\tau)\\
&\le 2\frac{\Delta s_2(T_l)}{s_1(T_l)}s_1(n+\tau)\left(1-C_12^\alpha s_1^{2\alpha}(T_l)\,(n+\tau-T_l)\right)^{-\beta}.
\end{aligned}
$$
For $i=1,2$ and $\min\{s_i, \tilde{s}_i\}\le \xi_i\le \max\{s_i, \tilde{s}_i\}$ and $\Delta s_i=\tilde{s_i}-s_i$ we have
$$s_i-|\Delta s_i|\le \xi_i\le s_i+|\Delta s_i|.$$
From this, \eqref{eq:chi-estimate} and \eqref{eq:Delta} it follows that
$$
\xi_1^2 + \xi_2^2\ge \xi_1^2\ge (s_1 - |\Delta s_1|)^2\ge (s_1-|\Delta s_2|)^2\ge s_1^2\left(1-\frac{|\Delta s_2|}{s_1}\right)^2 \ge C^{-1}s_1^2
$$
and
$$
s_1^2\le\xi_1^2 + \xi_2^2\le
2 (s_1+|\Delta s_2|)^2\le s_1^2\left(1+\frac{|\Delta s_1|}{s_1}\right)^2 \le Cs_1^2.
$$
Hence, for all $0<\alpha<1$ we obtain
$$
(\xi_1^2(n+\tau)+\xi_2^2(n+\tau))^{\alpha-\frac12}\le C s_1^{2\alpha -1}(n+\tau).
$$
Applying Lemma~\ref{s2estimate} on the time interval $[m_l, n+1]$ and using \eqref{AnBn} and the above estimates, we get 
$$
\begin{aligned}
\|A_n - B_n\|
&\le C \sup_{0\le\tau\le 1}\Big[s_1^{2\alpha-1}(n+\tau)\|\Delta s(n+\tau)\|\Big]\\
&\le 2C \frac{\Delta s_2(T_l)}{s_1(T_l)}\sup_{0\le\tau\le 1}\Big[s_1^{2\alpha}(n+\tau)
(1-C_12^\alpha s_1^{2\alpha}(T_l)\,(n+\tau-T_l))^{-\beta} \Big].
\end{aligned} 
$$
By \eqref{eq:chi-estimate}, we get
$$
\|A_n-B_n\| \le 
C \frac{|\Delta s_2 (T_l)|}{s_1(T_l)}\,\sup_{0\le\tau\le 1}\Big[s_1^{2\alpha}(T_l)
\Big(1-C_12^\alpha s_1^{2\alpha}(T_l)\,(n+\tau-T_l)\Big)^{-\beta-1}\Big].
$$
Since $s_1(T_l)=s_2(T_l)$, it follows from Lemma~\ref{bad-delta} that 
$$
\|A_n-B_n\| \le 
C \frac{|\Delta s_2 (m_l)|}{s_2(m_l)}\,s_1^{2\alpha}(T_l)
\Big(1-C_12^\alpha s_1^{2\alpha}(T_l)\,(n+1-T_l)\Big)^{-\beta-1}.
$$
Since $\frac{r_0}{4}\le s_2(m_l)\le r_0$ and both $|\Delta s_2 (m_l)|$ and 
$s_1^{2\alpha}(T_l)$ are uniformly bounded, the arguments similar to those in Case 1 yield
$$
\frac{\|A_n-B_n\|}{d(x_{n_{2l+1}}, y_{n_{2l+1}})}\le
C \Big(1-C_12^\alpha s_1^{2\alpha}(T_l)\,(n-T_l)\Big)^{-\beta-1}
$$
and we obtain
$$
\sum_{n=T_l}^{m_{l+1}} \frac{\|A_n - B_n\|}{ d(x_{n_{2l+1}}, y_{n_{2l+1}})}\le C.
$$ 
This completes the proof of the summability for $\delta_n$.

\vskip .1in
\noindent\textit{{\bf Part 2}: Estimating $\gamma_n$}. Observe that for all 
$n\in [n_{2l}, n_{2l+1}-1]$, the linear map contracts angles uniformly, and hence, $\gamma_n <\gamma <1$. For $n\in [m_l, m_{l+1}]$, where 
$x_n\in D_{r_0/2}$, Lemma~\ref{bad-gamma1} applies and yields 
$$ 
\prod_{n=m_l}^{m_{l+1}}\gamma_n \le (1 + C(m_{l+1}-m_l))^{-\frac{1}{\alpha}}
$$
for some constant $C>0$.
Since $[m_l, m_{l+1}]$ differs from $[n_{2l+1},n_{2l+2}]$ by a finite set, 
$$
\prod_{n=n_{2l+1}}^{n_{2l+2}-1}\gamma_n \le C'(1 + C(n_{2l+2}-n_{2l+1}))^{-\frac{1}{\alpha}}
$$
for some constant $C'>0$. In particular, we have 
$$ 
\prod_{n=n_{2l}}^{n_{2l+2}-1}\gamma_n \le C'\gamma^{n_{2l+1}-n_{2l}} < \theta_3 
$$
for some constant $0<\theta_3<1$, which implies the last estimate of Lemma~\ref{summable} with $\theta_2=\theta_3^l$. Moreover, 
\begin{multline*}
\sum_{n=0}^{\tau(x)} \prod_{j=0}^{n}\gamma_j =  \sum_{l=0}^{k(x)} \sum_{n=n_{2l}}^{n_{2l+2}-1}\prod_{j=0}^{n}\gamma_j 
\le\sum_{l=0}^{k(x)} \left(\prod_{j=0}^{n_{2l}-1}\gamma_j \sum_{n=n_{2l}}^{n_{2l+2}-1}\prod_{j=n_{2l}}^{n}\gamma_j \right) \\
\le\sum_{l=0}^{k(x)} \left(\theta_3^l \left( \sum_{n=n_{2l}}^{n_{2l+1}-1}\prod_{j=n_{2l}}^{n}\gamma_j + \prod_{j=n_{2l}}^{n_{2l+1}-1}\gamma_j\sum_{n=n_{2l+1}}^{n_{2l+2}-1}\prod_{j=n_{2l+1}}^{n}\gamma_j \right)  \right) \\
\le\sum_{l=0}^{k(x)} \left(\theta_3^l \left( \sum_{n=n_{2l}}^{n_{2l+1}-1}\gamma^{n-n_{2l}} + \theta_3\sum_{n=n_{2l+1}}^{n_{2l+2}-1}(1+C_1(n_{2l+2}-n)^{-\frac{1}{\alpha}} \right)  \right) .
\end{multline*}  
In the last line of the above formula, each of the two sums in the inner parenthesis is uniformly bounded and hence, for some $C''>0$
$$	
\sum_{n=0}^{\tau(x)-1}\prod_{j=0}^{n}\gamma_j\le C''\sum_{l=0}^{k(x)}\theta_3^l 
$$
is also uniformly bounded. This proves summability of $\gamma_n$ and completes the proof of Lemma~\ref{summable}.
\end{proof}

\section{Proof of Theorem~\ref{Katok1}}
\label{sec:proof}

Since the conjugacy map $H$ preserves topological and combinatorial information about $A$ (e.g., its topological entropy), the number $S_n$ of partition elements with inducing time $\tau_i=n$ for $G_{\T^2}$ and $A$ is the same and hence, by Lemma \ref{h1}, $S_n\le e^{hn}$ where 
$h<h_{\text{top}}(f)$. Observe that $h<h(\mu_1)$, where $\mu_1=m$ is the area. Indeed, this is the case for the automorphism $A$, and by Statement 4 of Proposition~\ref{smoothH}, the same holds for $G_{\T^2}$ provided $r_0$ is  sufficiently small. Proposition~\ref{geom_poten} and the fact that the inducing time is the first return time imply the existence of $t_0=t_0(P)<0$ such that for every $t_0<t<1$ there exists a unique equilibrium measure $\mu_t$ associated to the geometric $t$-potential $\varphi_t$ among all measures $\mu$ for which $\mu(P)>0$. The measure $\mu_t$ is ergodic, has exponential decay of correlations and satisfies the CLT with respect to a class of functions which includes all H\"older continuous functions on the torus. Note that $\mu_t(U)>0$ for every open set $U\subset P$.

Since the linear map $A$ has the Bernoulli property, every power of $A$ is ergodic. This implies that the tower for $A$ satisfies the arithmetic condition and hence, this is true for the tower for the Katok map $G_{\T^2}$. 

To prove the requirement \eqref{expansion} note that if $x,y\in\Lambda_i^s$ and 
$y\in\gamma^s(x)$, the distance $d(f^j(x), f^j(y))$ is decreasing with $j$ (see \eqref{lineardistance} and \eqref{distance}) and if $y\in\gamma^u(x)$, the distance 
$d(f^j(x), f^j(y))$ is increasing with $j$ reaching its maximum $\le\text{diam }P$ at 
$j=\tau(x)$. 

We now show that $t_0$ tends to $-\infty$ as $r_0$ approaches $0$. To this end, we use the formula \eqref{number-t0} and prove that the number $\log\lambda_1$ given by \eqref{eq:lambda1} can be chosen arbitrarily close to $h_{\mu_1}(f)$ for sufficiently small $r_0$. 

To this end let us fix $\varepsilon>0$ and chose a point $x\in\Lambda^s_i$. 
Note that $\phi=Id$ outside $D_{r_0}$ and does not change the stable and unstable directions at points $x\not\in D_{r_0}$. In view of relation~\eqref{eq:Fconjugacy} it suffices to estimate $\log\lambda_1$ working with the map $G$ instead of the Katok map $G_{\T^2}$. We can assume that $x$ is a generic point for the area $m$, which coincides with the measure $\mu_1$. We write 
$$
\tau_i=\sum_{j=1}^s n_j,
$$
where the numbers $n_j$ are chosen in the following way: 1) the number $n_1$ is the first moment when $G^{n_1}(x)\in D_{r_0}\setminus D_{\frac{r_0}{8}}$; 2) the number $n_1+n_2$ is the first moment when $G^{n_1+n_2}(x)\in D_{\frac{r_0}{8}}$; 3) the number $n_1+n_2+n_3$ is the first moment when 
$G^{n_1+n_2+n_3}(x)\in D_{r_0}\setminus D_{\frac{r_0}{8}}$; 4) the number $n_1+n_2+n_3+n_4$ is the first moment when 
$G^{n_1+n_2+n_3+n_4}(x)\notin D_{r_0}$; and continue in the same fashion.\footnote{Of course, some of the numbers $n_j$ can be zero but this does not affect the argument.} Note that $n_1\ge Q$ where the number $Q$ is given by \eqref{partition}. If $r_0$ is sufficiently small, then $Q$ becomes large enough to ensure that 
\begin{equation}\label{estim1}
\log |\J G^{n_1}(x)|\le n_1(\log\lambda +\varepsilon).
\end{equation}
By Equation~\eqref{eq:D}, for $x\in D_{r_0}\setminus D_{\frac{r_0}{8}}$ we have 
$\log|\J G(x)|\le \log M$ for some constant $M$ independent of $r_0$ and hence, 
\begin{equation}\label{estim2}
\log |\J G^{n_2}(x)|\le n_2\log M.
\end{equation}
For $x\in D_{\frac{r_0}{8}}$, we have that $\psi(u)=\left(\frac{u}{r_0}\right)^{\alpha}$  and $\psi'(u)=\frac{\alpha}{r_0}\left(\frac{u}{r_0}\right)^{\alpha-1}$ and by Equation~\eqref{eq:D}, we have $\log|\J G(x)|\le\log\lambda$. This implies that
\begin{equation}\label{estim3}
\log |\J G^{n_3}(x)|\le n_3\log\lambda.
\end{equation}
Finally, as in \eqref{estim2},
\begin{equation}\label{estim4}
\log |\J G^{n_4}(x)|\le n_4 \log M.
\end{equation}
Similar estimates hold for other $n_j$. It is easy to see that
\begin{equation}\label{estim5}
\log |\J F(x)|\le\sum_{j=1}^s\log |\J f^{n_1+\cdots +n_j}(f^{n_1+\cdots +n_{j-1}}(x))|.
\end{equation}
As in Lemma~\ref{tran-bound}, the maximal number of subsequent iterates (under $G$) any orbit spends in $D_{r_0}\setminus D_{\frac{r_0}{8}}$ is bounded from above by a constant $T'_0$ which does not depend on $r_0$. It follows from \eqref{estim1}-\eqref{estim5} that
$$
\log\lambda_1\le\log\lambda+\varepsilon+\frac{2T'_0\log M}{Q}\le\log\lambda+2\varepsilon
$$

On the other hand, by Statement 4 of Proposition~\ref{smoothH}, we can choose $r_0$ so small that 
$$
\log\lambda+\varepsilon\ge h_m(G_{\T^2})\ge\log\lambda-\varepsilon
$$ 
By Remark \ref{estim0}, we have that 
$$
\log\lambda+2\varepsilon\ge\log\lambda_1\ge h_m(G_{\T^2})\ge\log\lambda-\varepsilon
$$
and hence, the difference $\log\lambda_1-h_m(G_{\T^2})$ can be made arbitrary small if $r_0$ is sufficiently small. 

Let us choose another element $\tilde P$ of the Markov partition which satisfies Condition~\eqref{partition}. Repeating the above argument there exists $\tilde{t}_0=t_0(\tilde{P})<0$ such that for every $\tilde{t}_0<t<1$ there exists a unique equilibrium measure $\tilde{\mu}_t$ associated to the geometric $t$-potential among all measures $\mu$ for which $\mu(\tilde{P})>0$, and $\tilde{\mu}_t(\tilde{U})>0$ for every open set $\tilde{U}\subset\tilde{P}$. Since the map $f$ is topologically transitive, for every open sets $U\subset P$ and $\tilde{U}\subset\tilde{P}$ there exists an integer $k$ such that 
$G_{\T^2}^k(U)\cap\tilde{U}\ne\emptyset$. Therefore, $\mu_t=\tilde{\mu}_t$. Note that if $r_0$ is sufficiently small, the union of partition elements that satisfy Condition~\eqref{partition} form a closed set $Z$ whose complement is a neighborhood of zero. The only measure which does not charge any element of the Markov partition, lying outside this neighborhood, is the Dirac measure $\delta_0$ at the origin. Clearly, $P(\delta_0)=0$. Set 
$$
t_0:=\max_{P\in\mathcal{P}, P\cap Z\ne\emptyset}t_0(P).
$$
It follows from what was said above that $t_0\to-\infty$ as $r_0$ approaches $0$. The first statement of the theorem now follows by observing that $P(\mu_t)>0$ for every $t_0<t<1$.

To prove the third statement of the theorem fix $t>1$ and choose an ergodic measure $\mu$ for $G_{\T^2}$. Observe that the positive Lyapunov exponent of $\mu$ is equal to 
$\int\,\log |J^u G_{\mathbb{T}^2}(x)|
\,d\mu(x)$. Using now the Margulis-Ruelle inequality for the entropy and Statement (2) of Proposition \ref{smoothH}, we find that for $t>1$
$$
h_\mu(f)\le\int_M\,\log|J^uG_{\T^2}(x)|\,d\mu(x)< t\int_M\,\log|J^uG_{\T^2}(x)|\,d\mu(x).
$$
It follows that 
$$
\begin{aligned}
h_\mu(f)&-t\int_M\,\log|J^uG_{\T^2}(x)|\,d\mu(x)<0\\
&=h_\mu(\delta_0)-t\int_M\,\log|J^uG_{\T^2}(x)|\,d\delta_0(x)
\end{aligned}
$$
and hence, the Dirac measure $\delta_0$ is the unique equilibrium measure for $\varphi_t$.

To prove the second statement of the theorem let $\mu$ be the equilibrium measure for $\varphi_1$. Then either $\mu$ has zero Lyapunov exponents, in which case it is the Dirac measure at the origin, or it has positive Lyapunov exponents, in which case by the entropy formula, it must be the area $m$. 

\bibliographystyle{alpha}
\bibliography{BiblioSenti2017}
\end{document}